\documentclass[reqno,11pt]{article}
\usepackage{a4wide}
\usepackage{color,eucal,enumerate,mathrsfs}
\usepackage[normalem]{ulem}
\usepackage{tikz-cd}
\usepackage{pdfsync}
\usepackage{amsmath}
\usepackage{amssymb,epsfig,bbm}
\numberwithin{equation}{section}
\usepackage{hyperref}
\usepackage{cite}
\usepackage{mathtools}

\usepackage[english]{babel}

\usepackage[latin1]{inputenc}

\newcommand{\B}{\mathbb{B}}

\newcommand{\N}{\mathbb{N}}
\newcommand{\Q}{\mathbb{Q}}
\newcommand{\R}{\mathbb{R}}

\newcommand{\mm}{{\mbox{\boldmath$m$}}}

\newcommand{\ppi}{{\mbox{\boldmath$\pi$}}}

\newcommand{\sppi}{{\mbox{\scriptsize\boldmath$\pi$}}}

\newcommand{\sfd}{{\sf d}}

\newcommand{\sfC}{{\sf C}}

\newcommand{\sfL}{{\sf L}}

\newcommand{\sfT}{{\sf T}}

\newcommand{\Id}{{\rm Id}}                          
\newcommand{\Kliminf}{K\kern-3pt-\kern-2pt\mathop{\rm lim\,inf}\limits}  
\newcommand{\supp}{\mathop{\rm supp}\nolimits}   
\renewcommand{\d}{{\mathrm d}}

\newcommand{\restr}[1]{\lower3pt\hbox{$|_{#1}$}}
\newcommand{\la}{{\langle}}                  
\newcommand{\ra}{{\rangle}}
\newcommand{\eps}{\varepsilon}  
\newcommand{\nchi}{{\raise.3ex\hbox{$\chi$}}}

\newcommand{\fr}{\penalty-20\null\hfill$\blacksquare$}                      

\newcommand{\e}{{\rm{e}}}                           

\renewcommand{\mm}{\mathfrak m}                                

\newenvironment{proof}{\removelastskip\par\medskip   
\noindent{\textit{Proof.}}\rm}{\penalty-20\null\hfill$\square$\par\medbreak}

\newtheorem{theorem}{Theorem}[section]

\newtheorem{corollary}[theorem]{Corollary}
\newtheorem{lemma}[theorem]{Lemma}
\newtheorem{proposition}[theorem]{Proposition}

\newtheorem{definition}[theorem]{Definition}

\newtheorem{remark}[theorem]{Remark}

\newcommand{\X}{{\sf X}}
\newcommand{\Y}{{\sf Y}}
\newcommand{\RCD}{{\sf RCD}}

\newcommand{\beq}{\begin{equation}}
\newcommand{\eeq}{\end{equation}}
\newcommand{\norm}[1]{[\![#1]\!]}
\newcommand{\testvf}{{\rm TestVF}(\ppi)}
\newcommand{\testvfn}{{\rm TestVF}_\N(\ppi)}
\newcommand{\testvfc}{{\rm TestVF}_c(\ppi)}
\newcommand{\testv}{{\rm TestV}}
\newcommand{\vf}{{\rm VF}(\ppi)}

\title{On the notion of parallel transport on $\sf RCD$ spaces}
\author{Nicola Gigli\thanks{SISSA, ngigli@sissa.it} \and Enrico Pasqualetto\thanks{SISSA, epasqual@sissa.it}}

\begin{document}
\maketitle
\begin{abstract}
We propose a general notion of parallel transport on $\sf RCD$ spaces, prove an unconditioned uniqueness result and existence under suitable assumptions on the space. 

\bigskip



\end{abstract}
\tableofcontents
\section{Introduction}
More than ten years ago Sturm \cite{Sturm06I,Sturm06II} and Lott-Villani \cite{Lott-Villani09} introduced the concept of lower Ricci curvature bounds for metric measure spaces. Their approach has been refined in \cite{AmbrosioGigliSavare11-2} and \cite{Gigli12} with the introduction of the class of metric measure spaces with Riemannian Ricci curvature bounded from below, $\RCD$ spaces in short, which is currently a very active research area. We refer to the surveys \cite{Villani2016,Villani2017} for an overview of the topic and detailed references.

Among the various recent contributions, of particular relevance for the current manuscript is the paper \cite{Gigli14} by the first author, where a second order calculus has been built. In particular, on $\RCD(K,\infty)$ spaces the covariant derivative of vector fields is well defined. Let us mention that in \cite{Gigli14} `vector field' is intended in the sense of abstract $L^2$-normed $L^\infty$-modules and that in our previous paper \cite{GP16} we showed  that on $\RCD(K,N)$ spaces this abstract notion has a canonically more concrete counterpart described in terms of pointed-measured-Gromov-Hausdorff limits of rescaled spaces (this uses the rectifiability results obtained in \cite{Mondino-Naber14} and in \cite{MK16},\cite{GP16-2}).

\bigskip

In the classical smooth Riemannian framework, covariant derivative and parallel transport are two closely related concepts, thus given the existence of covariant derivative on $\RCD$ spaces it is natural to ask: is there a notion of parallel transport in the same setting? In this paper we address this question, our main results being:
\begin{itemize}
\item[i)] We provide a precise framework and give a rigorous meaning to the `PDE' defining the parallel transport (see Definitions \ref{def:w12}, \ref{def:h12} and \ref{def:pt}).
\item[ii)] By the nature of our definition, norm-preservation and linearity of the parallel transport are easy to derive, and these in turn will give uniqueness (see Corollary \ref{cor:uni}).
\item[iii)] On $\RCD$ spaces satisfying a certain regularity property, we are able to show existence of the parallel transport (see Section \ref{se:ex}). The regularity condition that we use is closely related to the existence of Sobolev vector fields with bounded covariant derivative (see Definition \ref{def:good_basis} for the precise assumption).
\end{itemize}
We believe that in fact the parallel transport exists on any $\RCD$ space, but we are currently unable to get the full proof. An insight on why this should not be too easy to prove is the following: on a space where the parallel transport exists, the dimension of the tangent module must be constant (see Theorem \ref{thm:constdim}) and thanks to the aforementioned paper \cite{GP16} this would in turn imply that the dimension of the pmGH-limits of rescaled spaces is constant. This very same result has been extremely elusive in the context of Ricci-limit spaces and has been obtained only relatively recently by Colding-Naber in \cite{ColdingNaber12}. Therefore it would be perhaps too optimistic to hope that the language proposed in \cite{Gigli14} and here allows for an easy generalization of such `constant dimension' result to the $\RCD$ setting. In this direction we remark that the assumptions that we put in order to obtain existence of the parallel transport are rather ad hoc and not really interesting from the geometric perspective: the intent with our existence result  is just to show that the approach we propose is non-void. In our forthcoming paper \cite{GP18} we shall study the same problem for more interesting geometric objects like Alexandrov spaces.

Let us also mention that in the appendix (see Theorem \ref{thm:sobbase}) we prove that 
\begin{itemize}
\item[iv)] Any $\RCD(K,\infty)$ space admits a base of the tangent module made of vectors in $W^{1,2}_C(T\X)$.
\end{itemize}
That is: if we relax the condition of `bounded covariant derivative' present in $(iii)$ above into `covariant derivative in $L^2$', then every $\RCD$ space meets the requirement.

\bigskip

Let us now briefly describe our approach. The crucial idea is that we don't study the problem of parallel transport along a single Lipschitz curve, but rather we study the problem along $\ppi$-a.e.\ curve, where $\ppi$ is a Borel probability measure on the space $\Gamma(\X)$  of continuous curves such that:
\[
\begin{split}
&\text{$\ppi$ is concentrated on equi-Lipschitz curves,}\\
&\text{for some $C>0$ we have $(\e_t)_*\ppi\leq C\mm$ for every $t\in[0,1]$,}
\end{split}
\]
where $\mm$ is the given reference measure on our $\RCD$ space $\X$ and $\e_t:\Gamma(\X)\to\X$ is the evaluation map defined by $\e_t(\gamma):=\gamma_t$. Measures $\ppi$ of this sort are a special case of so-called \emph{test plans} introduced in \cite{AmbrosioGigliSavare11}; these  can be used to define Sobolev functions on metric measure spaces by a kind of duality argument. The advantage of working with these plans rather than with single curves is that they are naturally linked to Sobolev calculus and thus also to all the functional-analytic machinery built in \cite{Gigli14}.

Let us now pretend, for the sake of this introduction, that our space $\X$ is in fact a smooth Riemannian manifold. In this case a time dependent vector field $(v_t)$ along $\ppi$ is, roughly said, given by a choice, for $\ppi$-a.e.\  $\gamma$,  of time dependent vector fields $(v^\gamma_t)$ on $\X$. Then we say that $(v_t)$ is a parallel transport along $\ppi$ provided for $\ppi$-a.e.\ $\gamma$ the vector field $t\mapsto v^\gamma_t(\gamma_t)$ is a parallel transport along $\gamma$. This happens if and only if
\[
\text{ for $\ppi$-a.e.\ $\gamma$ we have }\qquad\partial_t v^\gamma_t+\nabla_{\gamma'_t}v^\gamma_t=0\quad a.e.\ t.
\]
A relevant part of our paper (the whole Chapter \ref{ch:funct}) is devoted to showing that the above PDE can be stated even in the non-smooth setting, the key point being that it is possible to define a closed operator acting on $L^2$ vector fields along $\ppi$ which plays the role of $(\partial_t+\nabla_{\gamma'_t})$, see Definitions \ref{def:convect}, \ref{def:w12} and Proposition \ref{prop:consistency}.

\medskip

We conclude this introduction recalling that Petrunin proved in \cite{Petrunin98} that a certain notion of parallel transport exists along geodesic on Alexandrov spaces. Uniqueness for his construction is still an open problem. The question of comparing our notion and his one is certainly interesting, but outside the scope of this manuscript.

\bigskip

\noindent{\bf Acknowledgments} This research has been supported by the MIUR SIR-grant `Nonsmooth Differential Geometry' (RBSI147UG4).

\section{Some basic notions}
To keep the presentation short, we shall assume the reader familiar with the language proposed in \cite{Gigli14} (see also \cite{Gigli17}).
\subsection{Curves  in Banach spaces}
We recall here some basic results about measurability and integration of Banach-valued maps
of a single variable $t\in[0,1]$.  A detailed discussion about this topic can be found e.g.\ in \cite{DiestelUhl77}.

\bigskip

Let $\B$ be a fixed Banach space. We will denote by $\B'$ its dual space. A \emph{simple function}
is any map $y:\,[0,1]\to\B$ that can be written in the form
\[
y=\sum_{i=1}^k\nchi_{E_i}\,v_i,
\quad\mbox{ for some }E_1,\ldots,E_k\in\mathscr{B}\big([0,1]\big)\mbox{ and }v_1,\ldots,v_k\in\B,
\]
where for any topological space $\X$ we denote by $\mathscr{B}(\X)$ the set of Borel subsets of $\X$. A map $y:\,[0,1]\to\B$ is said to be \emph{strongly measurable} provided there exists a sequence
${(y_n)}_n$ of simple functions $y_n:\,[0,1]\to\B$ such that $\lim_n{\big\|y_n(t)-y(t)\big\|}_\B=0$
for $\mathcal{L}^1$-a.e.\ $t\in[0,1]$,
while it is said to be \emph{weakly measurable} provided $[0,1]\ni t\mapsto\omega\big(y(t)\big)\in\R$
is a Borel map for every $\omega\in\B'$.
It directly follows from the very definition that linear combinations of strongly (resp.\ weakly) measurable
functions are strongly (resp.\ weakly) measurable. Moreover, if a map $y:\,[0,1]\to\B$ is strongly measurable,
then the function $[0,1]\ni t\mapsto{\big\|y(t)\big\|}_\B\in\R$ is Borel.

The relation between the strongly measurable functions and the weakly measurable ones is fully described
by a theorem of Pettis, which states that a function $y:\,[0,1]\to\B$ is strongly measurable
if and only if it is weakly measurable and there exists a Borel set $N\subseteq[0,1]$ of null $\mathcal{L}^1$-measure
such that $y\big([0,1]\setminus N\big)$ is a separable subset of $\B$.

We now describe how to define $\B$-valued integrals, the so-called \emph{Bochner integrals}.
First of all, given a simple function $y:\,[0,1]\to\B$, written as $y=\sum_{i=1}^k\nchi_{E_i}\,v_i$, we define
\[
\int_0^1 y(t)\,\d t:=\sum_{i=1}^k\mathcal{L}^1(E_i)\,v_i\in\B.
\]
It can be readily checked that this definition does not depend on the particular way of
expressing the function $y$.
Further, we say that any strongly measurable function $y:\,[0,1]\to\B$ is \emph{Bochner integrable} provided
there exists a sequence ${(y_n)}_n$ of simple functions $y_n:\,[0,1]\to\B$
such that $\lim_n\int_0^1{\big\|y_n(t)-y(t)\big\|}_\B\,\d t=0$.
In particular, the sequence ${\big(\int_0^1 y_n(t)\,\d t\,\big)}_n\subseteq\B$ is Cauchy,
so that it makes sense to define
\[
\int_0^1 y(t)\,\d t:=\lim_{n\to\infty}\int_0^1 y_n(t)\,\d t\in\B.
\]
It turns out that the value of $\int_0^1 y(t)\,\d t$ just defined is independent of the approximating
simple functions ${(y_n)}_n$ and that it satisfies the fundamental inequality
\begin{equation}\label{eq:fund_inequality_Bochner}
{\left\|\int_0^1 y(t)\,\d t\,\right\|}_\B\leq\int_0^1{\big\|y(t)\big\|}_\B\,\d t.
\end{equation}

An alternative characterisation of the $\B$-valued integrable maps is given by the following
theorem, which is due to Bochner: a strongly measurable function $y:\,[0,1]\to\B$ is Bochner
integrable if and only if it satisfies $\int_0^1{\big\|y(t)\big\|}_\B\,\d t<+\infty$.

The previous result naturally leads to the notion of $\B$-valued $L^p$ space:
given $p\in[1,\infty]$, we define $L^p\big([0,1],\B\big)$ as the space of all (equivalence classes of) those strongly
measurable maps $y:\,[0,1]\to\B$ for which the quantity ${\|y\|}_{L^p([0,1],\B)}$ is finite, where
\[
{\|y\|}_{L^p([0,1],\B)}:=\left\{\begin{array}{ll}
\displaystyle{\left(\int_0^1{\big\|y(t)\big\|}^p_\B\,\d t\right)^{1/p}}&\quad\mbox{ if }p<\infty,\\
\displaystyle{\underset{t\in[0,1]}{{\rm ess}\,\sup}\,{\big\|y(t)\big\|}_\B}&\quad\mbox{ if }p=\infty.
\end{array}\right.
\]
Hence $L^p\big([0,1],\B\big)$ itself is a Banach space, for any $p\in[1,\infty]$. 
\begin{definition}[Vector-valued Sobolev/absolutely continuous maps] Let $p\in[1,\infty]$. The space $W^{1,p}([0,1],\B)$ consists of those curves $y\in L^p([0,1],\B)$ for which there is $y'\in L^p([0,1],\B)$ such that
\[
\int_0^1\varphi'(t)y(t)\,\d t=-\int_0^1\varphi(t)y'(t)\,\d t\qquad\forall \varphi\in C^\infty_c(0,1).
\]
It is equipped with the norm
\[
\|y\|_{W^{1,p}([0,1],\B)}:=\Big(\|y\|^p_{L^{p}([0,1],\B)}+\|y'\|^p_{L^{p}([0,1],\B)}\Big)^{1/p}.
\]
The space $AC^p([0,1],\B)$ consists of those curves $y:[0,1]\to \B$ for which there is $f\in L^p(0,1)$ such that
\[
\|y(s)-y(t)\|_\B\leq\int_t^sf(r)\,\d r,\qquad\forall t,s\in[0,1],\ t\leq s.
\]
\end{definition}

\begin{proposition}[Absolutely continuous representative]
Let $y\in W^{1,p}([0,1],\B)$. Then there is $\tilde y\in AC^p([0,1],\B)$ such that $y(t)=\tilde y(t)$ for a.e.\ $t$. Moreover, such $\tilde y$ satisfies
\[
\tilde y(s)-\tilde y(t)=\int_t^sy'(r)\,\d r\qquad\forall t,s\in[0,1],\ t<s.
\]
\end{proposition}
\begin{proof} The curve $t\mapsto z(t):=\int_0^t y'(s)\,\d s$ belongs to $AC^p\cap W^{1,p}([0,1],\B)$ and thus in particular the curve $t\mapsto y(t)-z(t)$ belongs to $W^{1,p}([0,1],\B)$ and, obviously, has derivative a.e.\ equal to 0. Hence to conclude it is sufficient to show that any such curve is a.e.\ constant. This follows noticing that for any $\ell\in \B'$ the map $t\mapsto\ell(y(t)-z(t))$ is in $W^{1,p}([0,1])$ (by direct verification) and has derivative a.e.\ equal to 0.
\end{proof}

\begin{proposition}[Characterization of curves in $W^{1,p}({[}0,1{]},\B)$]\label{prop:carw1p}
Let $y,z\in L^p([0,1],\B)$. Then $y\in W^{1,p}([0,1],\B)$ and $z=y'$ if and only if for some dense set $D\subset \B'$ we have that $\ell\circ y\in W^{1,1}([0,1])$ with $(\ell\circ y)'=\ell\circ z$ a.e.\ for every $\ell\in D$.

In particular, if $\B=L^{\tilde p}(\mu)$ for some Radon measure $\mu$, then $y\in W^{1,p}([0,1],L^{\tilde p}(\mu))$ and $z=y'$ if and only if for every Borel set $E$ we have that $t\mapsto\int_Ey(t)\,\d\mu$ is in $W^{1,1}(0,1)$ with derivative given by $t\mapsto\int_Ez(t)\,\d\mu$.
\end{proposition}
\begin{proof}
By assumption and using the fact that the Bochner integral commutes with the application of $\ell$ we have that
\[
\ell\Big(\int_0^1\varphi'(t)y(t)\,\d t\Big)=\ell\Big(-\int_0^1\varphi(t)z(t)\,\d t\Big)\qquad\forall \varphi\in C^\infty_c(0,1),
\]
for every $\ell\in D$. The conclusion follows by the density of $D$ in $\B'$.

For the second claim just observe that the linear span of the set of characteristic functions of Borel sets is dense in $L^{\tilde q}(\mu)\sim (L^{\tilde p}(\mu))'$.
\end{proof}
It is important to underline that in general absolute continuity does not imply a.e.\ differentiability: this has to do with the so-called Radon-Nikodym property of the target Banach space. A sufficient condition for this implication to hold is given in the next theorem:
\begin{theorem}\label{thm:acw}
Let $\B$ be a reflexive Banach space, $p\in[1,\infty]$ and $y\in AC^p([0,1],\B)$. Then for a.e.\ $t\in[0,1]$ the limit of $\frac{y(t+h)-y(t)}{h}$ as $h\to 0$ exists in $\B$. 

In particular,  $AC^p([0,1],\B)\sim W^{1,p}([0,1],\B)$, i.e.\ every absolutely continuous curve is the (only) continuous representative of a curve in $W^{1,p}([0,1],\B)$.
\end{theorem}

Given any Banach space $\B$, we shall denote by ${\rm End}(\B)$ the space of all linear and
continuous maps of $\B$ to itself, which is a Banach space if endowed with
the operator norm.

The space $\Gamma(\B):=C\big([0,1],\B\big)$ is a Banach space with respect to the
norm ${\|\cdot\|}_{\Gamma(\B)}$, given by
${\|y\|}_{\Gamma(\B)}:=\max\big\{{\|y_t\|}_\B\,:\,t\in[0,1]\big\}$ for every $y\in\Gamma(\B)$.

\begin{theorem}[Integral solutions to vector-valued linear ODEs]\label{thm:linear_ODE_int}
Let $\B$ be a Banach space. Let $z\in\Gamma(\B)$.
Let $\lambda:\,[0,1]\to{\rm End}(\B)$ be a bounded function, i.e.\ there exists $c>0$
such that $\big\|\lambda(t)\big\|_{{\rm End}(\B)}\leq c$ for every $t\in[0,1]$.
Assume that $[0,1]\ni t\mapsto\lambda(t)v\in\B$ is strongly measurable for every $v\in\B$.
Then there exists a unique curve $y\in\Gamma(\B)$ such that
\begin{equation}\label{eq:linear_ODE_int}
y(t)=z(t)+\int_0^t\lambda(s)y(s)\,\d s\quad\mbox{ for every }t\in[0,1].
\end{equation}
Moreover, the solution $y$ satisfies ${\|y\|}_{\Gamma(\B)}\leq e^c\,{\|z\|}_{\Gamma(\B)}$.
\end{theorem}
\begin{proof}
Given any simple function $t\mapsto y_t=\sum_{i=1}^k\nchi_{A_i}(t)\,v_i$, with
$A_1,\ldots,A_k\in\mathscr{B}\big([0,1]\big)$ and $v_1,\ldots,v_k\in\B$,
we have that $t\mapsto\lambda(t)y_t=\sum_{i=1}^k\nchi_{A_i}(t)\,\lambda(t)v_i$ is strongly measurable
by hypothesis on $\lambda$. Now fix $y\in\Gamma(\B)$. In particular, $y:\,[0,1]\to\B$
is strongly measurable, hence there exists a sequence ${(y^k)}_k$ of simple functions
$y^k:\,[0,1]\to\B$ such that $\lim_k{\|y^k_t-y_t\|}_\B=0$ holds for $\mathcal{L}^1$-a.e.\ $t\in[0,1]$.
This grants that ${\big\|\lambda(t)y^k_t-\lambda(t)y_t\big\|}_\B\leq c\,{\|y^k_t-y_t\|}_\B
\overset{k}{\to}0$ is satisfied for $\mathcal{L}^1$-a.e.\ $t\in[0,1]$, thus accordingly the map
$t\mapsto\lambda(t)y_t$ is strongly measurable as pointwise limit of strongly measurable functions.
Moreover, since ${\big\|\lambda(t)y_t\big\|}_\B\leq c\,{\|y\|}_{\Gamma(\B)}$ for all $t\in[0,1]$,
one has that $t\mapsto\lambda(t)y_t$ actually belongs to $L^\infty\big([0,1],\B\big)$.
Therefore it makes sense to define the function $\Lambda y:\,[0,1]\to\B$ as
$(\Lambda y)(t):=\int_0^t\lambda(s)y_s\,\d s$ for every $t\in[0,1]$.
Note that
\begin{equation}\label{eq:linear_ODE_int_1}
{\big\|\Lambda y(t_1)-\Lambda y(t_0)\big\|}_\B
\leq c\,{\|y\|}_{\Gamma(\B)}(t_1-t_0)\quad\mbox{ for every }t_0,t_1\in [0,1]\mbox{ with }t_0<t_1.
\end{equation}
Then $\Lambda y$ is Lipschitz with ${\rm Lip}(\Lambda y)\leq c\,{\|y\|}_{\Gamma(\B)}$,
so in particular $\Lambda y\in\Gamma(\B)$.
By plugging $t_1=t$ and $t_0=0$ into \eqref{eq:linear_ODE_int_1}, we deduce that
${\big\|\Lambda y(t)\big\|}_\B\leq c\,{\|y\|}_{\Gamma(\B)}t$ for all $t\in[0,1]$ and accordingly
that ${\|\Lambda y\|}_{\Gamma(\B)}\leq c\,{\|y\|}_{\Gamma(\B)}$.
This guarantees that the mapping $\Lambda:\,\Gamma(\B)\to\Gamma(\B)$ is linear and continuous,
with ${\|\Lambda\|}_{{\rm End}(\Gamma(\B))}\leq c$.
Now observe that
\begin{equation}\label{eq:linear_ODE_int_2}
y\in\Gamma(\B)\mbox{ satisfies \eqref{eq:linear_ODE_int}}
\quad\Longleftrightarrow\quad
({\Id}_{\Gamma(\B)}-\Lambda)(y)=z.
\end{equation}
For any $n\in\N^+$, the iterated operator $\Lambda^n=\Lambda\circ\ldots\circ\Lambda$ satisfies
\[\begin{split}
{\big\|\Lambda^n y(t)\big\|}_{\B}&\leq
c\int_0^t{\big\|\Lambda^{n-1}y(t_n)\big\|}_{\B}\,\d t_n\\
&\leq c^2\int_0^t\!\!\int_0^{t_n}{\big\|\Lambda^{n-2}y(t_{n-1})\big\|}_{\B}
\,\d t_{n-1}\,\d t_n\\
&\leq\ldots\ldots\\
&\leq c^n\int_0^t\!\!\int_0^{t_n}\!\!\ldots\!\int_0^{t_2}{\big\|y(t_1)\big\|}_{\B}
\,\d t_1\ldots\d t_{n-1}\,\d t_n\\
&\leq c^n\,{\|y\|}_{\Gamma(\B)}\int_0^t\!\!\int_0^{t_n}\!\!\ldots\!\int_0^{t_2}
\,\d t_1\ldots\d t_{n-1}\,\d t_n\\
&=c^n\,{\|y\|}_{\Gamma(\B)}\,\frac{t^n}{n!}
\end{split}\]
for every $y\in\Gamma(\B)$ and $t\in[0,1]$, whence
${\|\Lambda^n\|}_{{\rm End}(\Gamma(\B))}\leq c^n/n!$.
Hence ${\rm Id}_{\Gamma(\mathbb{B})}-\Lambda$ is invertible and the operator norm of its inverse
$(\Id_{\Gamma(\B)}-\Lambda)^{-1}=\sum_{n=0}^\infty\Lambda^n$ is bounded above by $e^c$.
In light of \eqref{eq:linear_ODE_int_2}, we finally conclude that there exists a unique curve
$y\in\Gamma(\B)$ fulfilling \eqref{eq:linear_ODE_int},
namely $y:=(\Id_{\Gamma(\B)}-\Lambda)^{-1}(z)$, which also satisfies
${\|y\|}_{\Gamma(\B)}\leq e^c\,{\|z\|}_{\Gamma(\B)}$.
\end{proof}
We will actually make use of is the following consequence of Theorem \ref{thm:linear_ODE_int}.
\begin{corollary}[Differential solutions to vector-valued linear ODEs]\label{cor:linear_ODE_diff}
Fix a reflexive Banach space $\B$. Let $\overline{y}\in\B$.
Let $\lambda:\,[0,1]\to{\rm End}(\B)$ be a bounded function.
Suppose that the map $[0,1]\ni t\mapsto\lambda(t)v\in\B$ is strongly measurable for every $v\in\B$.
Then there exists a unique curve $y\in{\rm LIP}\big([0,1],\mathbb{B}\big)$ such that
\begin{equation}\label{eq:linear_ODE_diff}
\left\{\begin{array}{ll}
y'(t)=\lambda(t)y(t)\quad\mbox{ for }\mathcal{L}^1\mbox{-a.e.\ }t\in[0,1],\\
y(0)=\bar{y}.
\end{array}\right.
\end{equation}
Moreover, the solution $y$ satisfies
${\|y\|}_{\Gamma(\B)}\leq e^c\,{\|\overline{y}\|}_\B$, where $c:=\max_{t\in[0,1]}\big\|\lambda(t)\big\|_{{\rm End}(\B)}$.
\end{corollary}
\begin{proof}
Define $z(t):=\overline{y}$ for every $t\in[0,1]$ and consider the curve $y\in\Gamma(\B)$, whose existence
is granted by Theorem \ref{thm:linear_ODE_int}. For every $t,s\in[0,1]$, with $s<t$, we have that
$${\big\|y(t)-y(s)\big\|}_\B={\left\|\int_s^t\lambda(r)y(r)\,\d r\right\|}_\B
\overset{\eqref{eq:fund_inequality_Bochner}}\leq\int_s^t{\big\|\lambda(r)y(r)\big\|}_\B\,\d r\leq
c\int_s^t{\big\|y(r)\big\|}_\B\,\d r.$$
Since ${\big\|y(\cdot)\big\|}_\B\in L^\infty(0,1)$, we deduce that the function $y$ is Lipschitz,
so that by Theorem \ref{thm:acw} $y$ is a.e.\ differentiable. Then  \eqref{eq:linear_ODE_diff} trivially follows from \eqref{eq:linear_ODE_int}.

Conversely, let $y\in{\rm LIP}\big([0,1],\B\big)$ be any curve such that \eqref{eq:linear_ODE_diff} holds true.
By integration we conclude that $y$
satisfies also property \eqref{eq:linear_ODE_int}, proving uniqueness.
\end{proof}
\subsection{Pullback of an \texorpdfstring{$L^0$}{L0}-normed module}\label{subsect:pullback_L0}
The aim of this subsection is to introduce the concept of pullback of an $L^0$-normed module
and to study its main properties. The following definitions and results mimic the analogous ones for
$L^p$-normed modules, which are treated in \cite[Subsection 1.6]{Gigli14}; a digression similar to the one below has been done in \cite{GR17}.
\begin{theorem}\label{thm:pullback_L0}
Let $({\rm X},\mathcal{A}_{\rm X},\mm_{\rm X})$, $({\rm Y},\mathcal{A}_{\rm Y},\mm_{\rm Y})$ be
$\sigma$-finite measured spaces. Let $\varphi:\,{\rm X}\to{\rm Y}$ be a map of bounded compression (i.e.\ $\varphi_*\mm_{\rm X}\leq C\mm_{\rm Y}$ for some $C>0$). Let
$\mathscr{M}^0$ be an $L^0(\mm_{\rm Y})$-normed module.
Then there exists (up to unique isomorphism) a unique couple $(\mathscr{N}^0,\sfT)$, where $\mathscr{N}^0$
is an $L^0(\mm_{\rm X})$-normed module and $\sfT:\,\mathscr{M}^0\to\mathscr{N}^0$ is a linear map, such that
\begin{itemize}
\item[$\rm(i)$] $|\sfT v|=|v|\circ\varphi$ holds $\mm_{\rm X}$-a.e.\ in $\rm X$,
for every $v\in\mathscr{M}^0$,
\item[$\rm(ii)$] the set of all the elements of the form $\sum_{i=1}^n\nchi_{A_i}\sfT v_i$,
with $(A_i)_{i=1}^n\subseteq\mathcal{A}_{\rm X}$ partition of $\rm X$ and $v_1,\ldots,v_n\in\mathscr{M}^0$, is dense
in $\mathscr{N}^0$.
\end{itemize}
Namely, if two couples $(\mathscr{N}^0_1,\sfT_1)$ and $(\mathscr{N}^0_2,\sfT_2)$ as above fulfill both
$\rm (i)$ and $\rm (ii)$, then there exists a unique module isomorphism $\Phi:\,\mathscr{N}^0_1\to\mathscr{N}^0_2$
such that the diagram
\[
\begin{tikzcd}
\mathscr{M}^0 \arrow{r}{\sfT_1} \arrow[swap]{rd}{\sfT_2} & \mathscr{N}^0_1 \arrow{d}{\Phi} \\
 & \mathscr{N}^0_2
\end{tikzcd}
\]
is commutative.
\end{theorem}
\begin{proof} \textbf{Existence.}
Fix $p\in[1,\infty)$ and let $\mathscr{M}:=\big\{v\in\mathscr{M}^0\,:\,|v|\in L^p(\mm_{\rm Y})\big\}$.
Hence $\mathscr{M}$ is an $L^p(\mm_{\rm Y})$-normed module and $\mathscr{M}^0$ is
the $L^0$-completion of $\mathscr{M}$.
Define $\mathscr{N}^0:=(\varphi^*\mathscr{M})^0$.
We construct the map $\sfT$ in the following way:
since the space $\mathscr{M}$ is dense in $\mathscr{M}^0$
and $\varphi^*\mathscr{M}$ is continuously embedded into $\mathscr{N}^0$,
the map $\varphi^*:\,\mathscr{M}\to\varphi^*\mathscr{M}$ can be uniquely extended to a linear continuous
function $\sfT:\,\mathscr{M}^0\to\mathscr{N}^0$. Such function $\sfT$ also
satisfies (i). Moreover, it is clear that the diagram
\[
\begin{tikzcd}
\mathscr{M} \arrow[hookrightarrow]{r} \arrow[swap]{d}{\varphi^*} & \mathscr{M}^0 \arrow{d}{\sfT} \\
\varphi^*\mathscr{M} \arrow[hookrightarrow]{r} & \mathscr{N}^0
\end{tikzcd}
\]
commutes. Therefore, since the set $\{\varphi^*v\,:\,v\in\mathscr{M}\}$ generates $\mathscr{N}^0$
as $L^0(\mm_{\rm X})$-normed module, we have in particular
that the set $\{\sfT v\,:\,v\in\mathscr{M}^0\}$ generates $\mathscr{N}^0$ as $L^0(\mm_{\rm X})$-normed module,
proving (ii).\newline
\textbf{Uniqueness.}
Let us choose $(\mathscr{N}^0_1,\sfT_1)$, $(\mathscr{N}^0_2,\sfT_2)$ satisfying $\rm (i)$ and $\rm (ii)$.
For $j=1,2$, denote by $V_j$ the set of $\sum_{i=1}^n\nchi_{A_i}\sfT_j v_i$ as in $\rm (ii)$.
Then the unique $L^0(\mm_{\rm X})$-linear map $\Psi:\,V_1\to V_2$, which
satisfies $\Psi\circ\sfT_1=\sfT_2$, is necessarily given by
$\Psi\big(\sum_{i=1}^n\nchi_{A_i}\sfT_1 v_i\big)=\sum_{i=1}^n\nchi_{A_i}\sfT_2 v_i$.
By requiring the condition $\rm (i)$, we force the $\mm_{\rm X}$-a.e.\ equality
$$\left|\sum_{i=1}^n\nchi_{A_i}\sfT_2 v_i\right|=\sum_{i=1}^n\nchi_{A_i}\,|v_i|\circ\varphi=
\left|\sum_{i=1}^n\nchi_{A_i}\sfT_1 v_i\right|,$$
which shows that the map $\Psi$ is actually well-defined and continuous.
There exists a unique module morphism $\Phi:\,\mathscr{N}^0_1\to\mathscr{N}^0_2$ that extends $\Psi$,
by density of $V_1$ in $\mathscr{N}^0_1$.
Such map $\Phi$ clearly satisfies $\Phi\circ\sfT_1=\sfT_2$. Finally, by interchanging the roles
of $\mathscr{N}^0_1$ and $\mathscr{N}^0_2$, one can easily conclude that $\Phi$ is an isomorphism, getting the thesis.
\end{proof}
\begin{definition}[Pullback module]
Any couple $(\mathscr{N}^0,\sfT)$ that satisfies Theorem \ref{thm:pullback_L0}
will be unambiguously denoted by $\big(\varphi^*\mathscr{M}^0,\varphi^*\big)$.
Moreover, we shall call $\varphi^*\mathscr{M}^0$ the \emph{pullback module of $\mathscr{M}^0$}
and $\varphi^*$ the \emph{pullback map}.
\end{definition}
\begin{proposition}[Universal property of the pullback]\label{prop:pullback_L0_univ_prop}
Let $({\rm X},\mathcal{A}_{\rm X},\mm_{\rm X})$, $({\rm Y},\mathcal{A}_{\rm Y},\mm_{\rm Y})$ be
$\sigma$-finite measured spaces. Let $\varphi:\,{\rm X}\to{\rm Y}$ be a map of bounded compression. Let
$\mathscr{M}^0$ be an $L^0(\mm_{\rm Y})$-normed module and let $\mathscr{N}^0$ be an $L^0(\mm_{\rm X})$-normed module.
Consider a linear operator $\sfT:\,\mathscr{M}^0\to\mathscr{N}^0$ such that
\begin{equation}\label{eq:univ_prop_1}
|\sfT v|\leq\ell\,|v|\circ\varphi\;\;\;\mm_{\rm X}\mbox{-a.e.\ in }{\rm X},
\quad\mbox{ for every }v\in\mathscr{M}^0,
\end{equation}
for a suitable map $\ell\in L^0(\mm_{\rm X})$. Then there exists a unique $L^0(\mm_{\rm X})$-linear and continuous
operator $\widehat{\sfT}:\,\varphi^*\mathscr{M}^0\to\mathscr{N}^0$ such that $\widehat{\sfT}\circ\varphi^*=\sfT$ and
\begin{equation}\label{eq:univ_prop_2}
|\widehat{\sfT}w|\leq\ell\,|w|\;\;\;\mm_{\rm X}\mbox{-a.e.\ in }{\rm X},
\quad\mbox{ for every }w\in\varphi^*\mathscr{M}^0.
\end{equation}
\end{proposition}
\begin{proof}
Denote by $V$ the set of all elements of the form $\sum_{i=1}^n\nchi_{A_i}\,\varphi^*v_i$,
with $A_1,\ldots,A_n\in\mathcal{A}_{\rm X}$ partition of $\rm X$ and $v_1,\ldots,v_n\in\mathscr{M}^0$, so that $V$ is
a dense linear subspace of $\varphi^*\mathscr{M}^0$ by property (ii) of Theorem \ref{thm:pullback_L0}.
Any $L^0(\mm_{\rm X})$-linear map $\widehat{\sfT}:\,\varphi^*\mathscr{M}^0\to\mathscr{N}^0$
with $\widehat{\sfT}\circ\varphi^*=\sfT$ must satisfy
\begin{equation}\label{eq:univ_prop_3}
\widehat{\sfT}w=\sum_{i=1}^n\nchi_{A_i}\widehat{\sfT}(\varphi^*v_i)=\sum_{i=1}^n\nchi_{A_i}\sfT v_i
\quad\mbox{ for }w=\sum_{i=1}^n\nchi_{A_i}\,\varphi^*v_i\in V.
\end{equation}
Consider $\widehat{\sfT}:\,V\to\mathscr{N}^0$ defined as in \eqref{eq:univ_prop_3}, then \eqref{eq:univ_prop_1} grants that
\begin{equation}\label{eq:univ_prop_4}
|\widehat{\sfT}w|=\sum_{i=1}^n\nchi_{A_i}\,|\sfT v_i|\leq\ell\sum_{i=1}^n\nchi_{A_i}\,|v_i|\circ\varphi
=\ell\sum_{i=1}^n\nchi_{A_i}\,|\varphi^*v_i|=\ell\,|w|\quad\mbox{ holds }\mm_{\rm X}\mbox{-a.e.,}
\end{equation}
which shows that $\widehat{\sfT}:\,V\to\mathscr{N}^0$ is well-defined (in the sense that $\widehat{\sfT}w$ depends only
on $w$ and not on the way of representing it) and continuous.
Therefore $\widehat{\sfT}$ can be uniquely extended to a linear continuous
operator $\widehat{\sfT}:\,\varphi^*\mathscr{M}^0\to\mathscr{N}^0$. We readily deduce from the
definition \eqref{eq:univ_prop_3} that the equality $f\,\widehat{\sfT}w=\widehat{\sfT}(fw)$ holds for
$f:\,\rm X\to\R$ simple function, so that $\widehat{\sfT}$ can be shown to be $L^0(\mm_{\rm X})$-linear by
an approximation argument. Finally, it follows from \eqref{eq:univ_prop_4} that the inequality
\eqref{eq:univ_prop_2} is satisfied for $w\in V$, whence \eqref{eq:univ_prop_2} holds
by density of $V$ in $\varphi^*\mathscr{M}^0$.
\end{proof}
\subsection{Some properties of test plans}

For the sake of brevity, hereafter we shall use the notation $\mathcal{L}_1$ to indicate
the $1$-dimensional Lebesgue measure restricted to $[0,1]$, namely
\[
\mathcal{L}_1:=\mathcal{L}^1\restr{[0,1]}.
\]
Let $\ppi\in\mathscr{P}\big(\Gamma({\rm X})\big)$ be any fixed test plan on $\rm X$,
whence $\big(\Gamma({\rm X}),\mathsf{d}_{\Gamma({\rm X})},\ppi\big)$ is a metric measure space.
Given that the map $\e_t$ is of bounded compression, it makes sense to consider the pullback module
$\e_t^* L^2(T{\rm X})$. Observe that $\e_t^*L^2(T{\rm X})$ is a Hilbert module
as soon as $({\rm X},\sfd,\mm)$ is infinitesimally Hilbertian.
\begin{remark}\label{rmk:f_circ_e_t_Borel}{\rm
Let us define the map $\e:\,\Gamma({\rm X})\times[0,1]\to{\rm X}$ as $\e(\gamma,t):=\gamma_t$
for every $\gamma\in\Gamma({\rm X})$ and $t\in[0,1]$. It can be easily proved that the map $\e$ is continuous. This grants
that, given any Borel map $f:\,{\rm X}\to\R$, the function $f\circ\e$ is Borel.
Moreover, observe that
$$(\ppi\times\mathcal{L}_1)\big(\e^{-1}(A)\big)=\int_0^1\ppi\big(\e_t^{-1}(A)\big)\,\d t
\leq\sfC(\ppi)\,\mm(A)\quad\mbox{ for every }A\in\mathscr{B}({\rm X})$$
by Fubini theorem, in other words it holds that $\e_*(\ppi\times\mathcal{L}_1)\leq\sfC(\ppi)\,\mm$.
Therefore one has that the composition $f\circ\e\in L^0(\ppi\times\mathcal{L}_1)$ is well-defined for any $f\in L^0(\mm)$.
\fr}\end{remark}
\begin{theorem}\label{thm:cont_f_circ_e_t}
Let $({\rm X},\sfd,\mm)$ be a metric measure space. Let $\ppi$ be a test plan on $\rm X$. Then
\begin{equation}\label{eq:cont_f_circ_e_t_1}
\mbox{ for every }
f\in L^1(\mm)\mbox{ the map }
[0,1]\ni t\longmapsto f\circ\e_t\in L^1(\ppi)\;\mbox{ is continuous}.
\end{equation}
In particular, the map $[0,1]\ni t\mapsto\int f\circ\e_t\,\d\ppi$ is
continuous for every $f\in L^1(\mm)$.
\end{theorem}
\begin{proof}
First of all, we claim that
\begin{equation}\label{eq:cont_f_circ_e_t_2}
\lim_{s\to t}\int|f\circ\e_s-f\circ\e_t|\,\d\ppi=0\quad\mbox{ if }f
\in C_b({\rm X})\cap L^1(\mm)\mbox{ and }t\in[0,1].
\end{equation}
To prove it, note that $|f\circ\e_s-f\circ\e_t|(\gamma)\leq 2\,{\|f\|}_{L^\infty(\mm)}$
for every $\gamma\in\Gamma({\rm X})$ and $t,s\in[0,1]$.
Moreover, $\big|f(\gamma_s)-f(\gamma_t)\big|\to 0$ as $s\to t$ by continuity, for every $\gamma\in\Gamma({\rm X})$
and $t\in[0,1]$. Hence we obtain \eqref{eq:cont_f_circ_e_t_2} as a consequence of the dominated convergence theorem.
Observe also that
\begin{equation}\label{eq:cont_f_circ_e_t_3}
L^1(\mm)\ni f\longmapsto f\circ\e_t\in L^1(\ppi)\;
\mbox{ is a linear bounded map,}\quad\mbox{ for every }t\in[0,1].
\end{equation}
Indeed, $\int|f\circ\e_t|\,\d\ppi\leq\sfC(\ppi)\int|f|\,\d\mm$ is satisfied for every $f\in L^1(\mm)$.
Now fix $f\in L^1(\mm)$. Choose a sequence ${(f_n)}_n\subseteq C_b({\rm X})\cap L^1(\mm)$ that
converges to $f$ with respect to the $L^1$-norm. Given $t\in[0,1]$ and $n\in\N$,
we have that \eqref{eq:cont_f_circ_e_t_2} and \eqref{eq:cont_f_circ_e_t_3} yield
\begin{equation}\label{eq:cont_f_circ_e_t_4}\begin{split}
\underset{s\to t}{\overline{\lim}}\int|f\circ\e_s-f\circ\e_t|\,\d\ppi
&\leq 2\,\sfC(\ppi)\,{\|f-f_n\|}_{L^1(\mm)}+
\underset{s\to t}{\overline{\lim}}\int|f_n\circ\e_s-f_n\circ\e_t|\,\d\ppi\\
&= 2\,\sfC(\ppi)\,{\|f-f_n\|}_{L^1(\mm)}.
\end{split}\end{equation}
By letting $n\to\infty$ in \eqref{eq:cont_f_circ_e_t_4}, we finally conclude that
$\int|f\circ\e_s-f\circ\e_t|\,\d\ppi\to 0$ as $s\to t$, which proves \eqref{eq:cont_f_circ_e_t_1}.
The last statement is obvious.
\end{proof}

Under further assumptions on $({\rm X},\sfd,\mm)$, we have at disposal also a notion
of `speed' $\ppi'_t$ of the test plan $\ppi$ at time $t$, as described in the following result.
For the proof of such fact, we refer to \cite[Theorem 2.3.18]{Gigli14} or \cite[Theorem/Definition 1.32]{Gigli17}.
\begin{theorem}[Speed of a test plan]\label{thm:speed_test_plan}
Let $({\rm X},\sfd,\mm)$ be a metric measure space such that $L^2(T{\rm X})$ is separable.
Let $\ppi$ be a test plan on $\rm X$. Then there exists a unique
(up to $\mathcal{L}_1$-a.e.\ equality) family $\ppi'_t\in\e_t^*L^2(T{\rm X})$ such that
\begin{equation}\label{eq:speed_test_plan_1}
\exists\,L^1(\ppi)\text{-}\lim_{h\to 0}\frac{f\circ\e_{t+h}-f\circ\e_t}{h}=
(\e_t^*\d f)(\ppi'_t)\quad\mbox{ for }\mathcal{L}^1\mbox{-a.e.\ }t\in[0,1],
\end{equation}
for every $f\in W^{1,2}({\rm X})$. Moreover, the function $(\gamma,t)\mapsto|\ppi'_t|(\gamma)$ is
(the equivalence class of) a Borel map such that for $\mathcal{L}^1$-a.e.\ $t\in[0,1]$ it holds that
\begin{equation}\label{eq:speed_test_plan_2}
|\ppi'_t|(\gamma)=|\dot{\gamma}_t|\quad\mbox{ for }\ppi\mbox{-a.e.\ }\gamma\in AC\big([0,1],{\rm X}\big).
\end{equation}
\end{theorem}
\begin{proposition}\label{prop:perborelsp}
Let $({\rm X},\sfd,\mm)$ be a metric measure space such that  $L^2(T{\rm X})$ is separable, let $\ppi$ be a test plan on $\rm X$ and $f\in W^{1,2}(\X)$. Then the a.e.\ defined map $[0,1]\ni t\mapsto \e_t^*\d f(\ppi'_t)\in L^1(\ppi)$ is a.e.\ equal to a Borel map.
\end{proposition}
\begin{proof}
For every $h\in(0,1)$ the map $[0,1-h]\ni t\mapsto(f\circ\e_{t+h}-f\circ\e_t)/h\in L^1(\ppi)$ is continuous. Thus by classical arguments the set of $t$'s for which the limit as $h\to 0$ exists is Borel and the limit function, set, say, to 0 when the limit does not exist, is Borel. 
\end{proof}

In this paper we shall work only on $\RCD$ spaces, which in particular are so that $W^{1,2}(\X)$ is reflexive. In turn this implies, by the arguments in \cite{ACM14}, that the tangent module is separable, so that the assumptions of Theorem \ref{thm:speed_test_plan} and Proposition \ref{prop:perborelsp} above are fulfilled.

In the sequel, we will mainly focus our attention on those test plans $\ppi$ that are concentrated on
an equiLipschitz family of curves. As illustrated in the next definition, we will refer to them as
`Lipschitz test plans'.
\begin{definition}[Lipschitz test plan]\label{def:lip_test_plan}
Let $({\rm X},\sfd,\mm)$ be a metric measure space such that  $L^2(T{\rm X})$ is separable.
Then a test plan $\ppi$ on $\rm X$ is said to be a
\emph{Lipschitz test plan} provided there exists a constant $\sfL\geq 0$ such that
\begin{equation}\label{eq:lip_test_plan}
|\ppi'_t|\leq\sfL\;\mbox{ holds }\ppi\mbox{-a.e.\ in }\Gamma({\rm X}),
\quad\mbox{ for }\mathcal{L}^1\mbox{-a.e.\ }t\in[0,1],
\end{equation}
or, equivalently, such that $\ppi$ is concentrated on the family
of all the $\sfL$-Lipschitz curves in $\rm X$.
The smallest constant $\sfL\geq 0$ for which \eqref{eq:lip_test_plan} is satisfied
will be denoted by $\sfL(\ppi)$.
\end{definition}

Whenever the test plan $\ppi$ is Lipschitz, we have $(\e_t^*\d f)(\ppi'_t)\in L^2(\ppi)$
for every $f\in W^{1,2}({\rm X})$. One is then led to wonder whether in this case the limit in \eqref{eq:speed_test_plan_1}
takes place not only in $L^1(\ppi)$, but also in $L^2(\ppi)$. The answer is affirmative,
as shown in the following simple result:
\begin{proposition}\label{prop:pi'_t_in_L2}
Let $({\rm X},\sfd,\mm)$ be a metric measure space such that  $L^2(T{\rm X})$ is separable.
Let $\ppi$ be a Lipschitz test plan on $\rm X$. Let $f\in W^{1,2}({\rm X})$.
Then the mapping $t\mapsto f\circ\e_t\in L^2(\ppi)$ is Lipschitz and
\begin{equation}\label{eq:pi'_t_in_L2_1}
L^2(\ppi)\text{-}\frac{\d}{\d t}(f\circ\e_t)=(\e_t^*\d f)(\ppi'_t)
\quad\mbox{ for }\mathcal{L}^1\mbox{-a.e.\ }t\in[0,1].
\end{equation}
\end{proposition}
\begin{proof}
Given any $t,s\in[0,1]$ with $s<t$, one has that
\[\begin{split}
{\big\|f\circ\e_t-f\circ\e_s\big\|}^2_{L^2(\sppi)}&=\int{\big|f(\gamma_t)-f(\gamma_s)\big|}^2\,\d\ppi(\gamma)\\
\text{(by definition of Sobolev functions)}\qquad&
\leq\int\left(\int_s^t|Df|(\gamma_r)\,|\dot{\gamma}_r|\,\d r\right)^2\d\ppi(\gamma)\\
\text{(by H\"{o}lder inequality)}\qquad&\leq(t-s)\,\sfL(\ppi)^2\int\!\!\!\int_s^t|Df|^2(\gamma_r)\,\d r\,\d\ppi(\gamma)\\
&\leq\sfC(\ppi)\,\sfL(\ppi)^2\,{\|f\|}^2_{W^{1,2}({\rm X})}\,(t-s)^2,
\end{split}\]
which shows that $t\mapsto f\circ\e_t\in L^2(\ppi)$ is a Lipschitz map. In particular, it is differentiable
at almost every $t\in[0,1]$ by Theorem \ref{thm:acw}, so that \eqref{eq:pi'_t_in_L2_1}
follows from \eqref{eq:speed_test_plan_1}.
\end{proof}

\section{Introduction of appropriate functional spaces}\label{ch:funct}

Throughout all this chapter, $(\X,\sfd,\mm)$ is a given $\RCD(K,\infty)$ space and $\ppi$ a test plan on it.

Recall that the space ${\rm TestF}(\X)$ of \emph{test functions on $\X$} is defined as
\[
{\rm TestF}(\X):=\Big\{f\in L^\infty\cap {\rm LIP}\cap W^{1,2}(\X)\cap D(\Delta)\ :\ \Delta f\in W^{1,2}(\X)\Big\}
\]
and that the space of \emph{test vector fields on $\X$} is defined as 
\[
{\rm TestV}(\X):=\Big\{\sum_{i=1}^nf_i\,\nabla g_i\ :\ n\in\N^+,\ f_i,g_i\in{\rm TestF}(\X) \text{ for every }i\Big\}\subset L^2(T\X).
\]

\subsection{Test vector fields along $\ppi$}\label{se:test}
We define the space of \emph{vector fields along $\ppi$} as:
\[
\vf:=\prod_{t\in[0,1]}\e_t^*L^2(T\X).
\]
Thus $\vf$ is the collection of maps assigning to each $t\in[0,1]$ an element of $\e_t^*L^2(T\X)$; it is a vector space w.r.t.\ pointwise operation. 

To each $V\in\vf$ we associate  the function $[\![V]\!]:\,[0,1]\to[0,+\infty)$, defined by
\[
{[\![V]\!]}_t:={\|V_t\|}_{\e_t^*L^2(T{\rm X})}\quad\text{ for every }t\in[0,1].
\]
The subspace $\testvf\subset\vf$ of \emph{test vector fields along $\ppi$} is defined as:
\[
\testvf:=\left\{t\mapsto\sum_{i=1}^n\varphi_i(t)\,\nchi_{A_i}\,\e_t^*v_i\;\bigg|\;
\begin{array}{cc}
n\in\N^+,\;A_i\in\mathscr{B}\big(\Gamma({\rm X})\big),
\;\varphi_i\in{\rm LIP}\big([0,1]\big),\\
\mbox{and }v_i\in{\rm TestV}({\rm X})\mbox{ for every }i=1,\ldots,n
\end{array}\right\}.
\]
Since ${\rm TestV}({\rm X})\subseteq L^\infty(T{\rm X})$ we  see that for any  $V\in\testvf$ the function $(\gamma,t)\mapsto|V_t|(\gamma)$ belongs to $L^\infty(\mathcal L_1\times\ppi)$.
\begin{proposition}[Continuity of the test vector fields along $\ppi$]\label{prop:test_curves_continuous}
For any $V,W\in\testvf$ we have that
\begin{equation}\label{eq:test_curves_continuous}
[0,1]\ni t\quad \to \quad \la V_t,W_t\ra\in L^1(\ppi)\quad\mbox{ is continuous.}
\end{equation}
In particular, the function $[\![V]\!]:[0,1]\to[0,+\infty)$ is continuous for every $V\in{\rm Test}\Gamma({\rm X})$.
\end{proposition}
\begin{proof}
By linearity, it is clear that it is sufficient to prove the claim for $V,W$ of the form $V=\nchi_A\e_t^*v$, $W=\nchi_B\e_t^*w$ for $v,w\in {\rm TestV}({\rm X})$. In this case the claim \eqref{eq:test_curves_continuous} is a direct consequence of 
\[
 \la V_t,W_t\ra=\nchi_{A\cap B}\la v,w\ra\circ\e_t
\]
and Theorem \ref{thm:cont_f_circ_e_t}. The last statement follows by choosing $W=V$.
\end{proof}
We now define two norms on $\testvf$:
\[
\begin{split}
{\|V\|}^2_{\mathscr{L}^2(\sppi)}&:=\int_0^1{[\![V]\!]}_t^2\,\d t\\
{\|V\|}_{\mathscr{C}(\sppi)}&:=\max_{t\in[0,1]} {[\![V]\!]}_t.
\end{split}
\]
Notice that Proposition \ref{prop:test_curves_continuous} ensures that $t\mapsto \norm{V}_t$ is Borel, hence $\|\cdot\|_{\mathscr L^2(\sppi)}$ is well defined; also, routine computations show that ${\|\cdot\|}_{\mathscr{L}^2(\sppi)},{\|\cdot\|}_{\mathscr{C}(\sppi)}$ are  norms on $\testvf$ with ${\|\cdot\|}_{\mathscr{L}^2(\sppi)}\leq {\|\cdot\|}_{\mathscr{C}(\sppi)}$. 

\bigskip

We now want to show that $(\testvf,\|\cdot\|_{\mathscr C(\sppi)})$ is separable by exhibiting  a countable dense subset. To this aim, we first  choose
three countable families $\mathcal{F}_1\subseteq\big\{\text{open sets of }\Gamma({\rm X})\big\}$,
$\mathcal{F}_2\subseteq{\rm LIP}\big([0,1]\big)$ and $\mathcal{F}_3\subseteq\testv({\rm X})$ such that
\[
\begin{split}
&\text{given }A\subseteq\Gamma({\rm X})\text{ Borel and }\eps>0,\text{ there exists }U\in\mathcal{F}_1
\text{ with }\ppi(A\Delta U)<\eps,\\
&\mathcal{F}_2\text{  is dense in }
C([0,1])\text{ and stable by product and $\Q$-linear combinations},\\
&\mathcal{F}_3\text{ is a  }\mathbb{Q}\text{-vector space of functions in }W^{1,2}(\X)
\text{ whose gradients generate }L^2(T{\rm X}).
\end{split}
\]
We proceed in the following way:
\begin{itemize}
\item[${\mathcal{F}}_1$:] Since $\Gamma({\rm X})$ is separable, there exists a countable family $\widetilde{\mathcal{F}}_1$
of open subsets of $\Gamma({\rm X})$ that is a neighbourhood basis for each point $\gamma\in\Gamma({\rm X})$.
Let us denote by $\mathcal{F}_1$ the set of finite unions of elements of $\widetilde{\mathcal{F}}_1$,
so that $\mathcal{F}_1$ is countable. Fix $A\in\mathscr{B}\big(\Gamma(\rm X)\big)$ and $\eps>0$.
The measure $\ppi$ is regular, since $\big(\Gamma({\rm X}),\sfd_{\Gamma({\rm X})}\big)$
is complete and separable. By inner regularity of $\ppi$, there exists a compact subset
$K\subseteq A$ such that $\ppi(A\setminus K)<\eps/2$.
By outer regularity of $\ppi$, there exists $V\subseteq\Gamma({\rm X})$
open such that $K\subseteq V$ and $\ppi(V\setminus K)<\eps/2$.
We can then associate to any $\gamma\in K$ a set
$U_\gamma\in\widetilde{\mathcal{F}}_1$ such that $\gamma\in U_\gamma\subseteq V$. By compactness of $K$, one has
$K\subseteq U_{\gamma_1}\cup\ldots\cup U_{\gamma_n}\subseteq V$ for some finite choice
$\gamma_1,\ldots,\gamma_n\in K$. Let us call $U:=U_{\gamma_1}\cup\ldots\cup U_{\gamma_n}\in\mathcal{F}_1$.
We thus have that
$$\ppi(A\Delta U)=\ppi(A\setminus U)+\ppi(U\setminus A)\leq\ppi(A\setminus K)+\ppi(V\setminus K)<\eps.$$
\item[${\mathcal{F}}_2$:] By the separability of  $C([0,1])$ such $\mathcal F_2$ exists.
\item[${\mathcal{F}}_3$:] Since the space $(\rm X,\sfd,\mm)$ is infinitesimally Hilbertian,
we have that $W^{1,2}({\rm X})$ is reflexive and therefore, by \cite{ACM14}, separable. Thus let ${\mathcal{F}}_3$ be any countable dense $\Q$-vector subspace of $W^{1,2}(\X)$ and notice that since gradients of functions in $W^{1,2}(\X)$ generate $L^2(T\X)$, the same holds for functions in ${\mathcal{F}}_3$.
\end{itemize}
We now define the class of test vector fields $\testvfn$ as:
\[
\testvfn:=
\left\{t\mapsto\sum_{i=1}^n\psi_i(t)\,\nchi_{U_i}\,\e_t^*\nabla f_i\;\bigg|\;
\begin{array}{cc}
n\in\N^+\mbox{ and }U_i\in\mathcal{F}_1,\,\psi_i\in\mathcal{F}_2,\\
f_i\in\mathcal{F}_3\mbox{ for every }i=1,\ldots,n
\end{array}\right\}.
\]
Clearly $\testvfn$ is a countable subset of $\testvf$. Also, notice that the inequalities
\[
\begin{split}
\norm{\nchi_A\e_t^*v-\nchi_A\e_t^*w}_t&\leq \sqrt{{\sf C}(\ppi)}\|v-w\|_{L^2(T\X)},\\
\norm{\nchi_A\e_t^*v-\nchi_U\e_t^*v}_t&\leq \sqrt{\ppi(A\Delta U)}\|v\|_{L^\infty(T\X)},
\end{split}
\]
valid for any $t\in[0,1]$, $A,U\subset\Gamma(\X)$ Borel and $v,w\in L^2(T\X)$ and the very definition of pullback module, show that 
\begin{equation}
\label{eq:densevfn}
\text{for any $t\in[0,1]$ the set $\{W_t\,:\,W\in\testvfn\}$ is dense in $\e_t^*L^2(T\X)$.}
\end{equation}

\begin{lemma}[Separability of $\testvf$]
\label{lem:separability_TestG(X)}
The family $\testvfn$ is dense in $\testvf$ with respect
to the norm ${\|\cdot\|}_{\mathscr{C}(\sppi)}$ (and thus also w.r.t.\ ${\|\cdot\|}_{\mathscr{L}^2(\sppi)}$).
\end{lemma}
\begin{proof} Let $V\in\testvf$ be arbitrary, $\eps>0$ and $t_0\in[0,1]$. By \eqref{eq:densevfn}  there is  $W\in \testvfn$ such that $\norm{V-W}_{t_0}<\eps$. Since 
$t\mapsto \norm{V-W}_t^2=\norm{V}^2_t+\norm{W}^2_t-2\int\la V_t,W_t\ra\,\d\ppi$ is continuous, we see that $\norm{V-W}_{t}<\eps$ for 
every $t$ in a neighbourhood of $t_0$. By compactness of $[0,1]$ we can then find a finite number of open intervals $I_1,\ldots,I_n$ covering $[0,1]$ and elements $W_1,\ldots,W_n\in \testvfn$ such that
\begin{equation}
\label{eq:b1}
\norm{V-W_i}_t< \eps\qquad\forall t\in I_i\cap[0,1],\ i=1,\ldots,n.
\end{equation}
By multiplying $W_i$ by an appropriate function in ${\mathcal{F}}_2$ we can also assume that
\begin{equation}
\label{eq:b2}
\norm{W_i}_t< \|V\|_{\mathscr C(\sppi)}+2\eps\qquad\forall t\in[0,1].
\end{equation}
Now let $(\phi_i)$ be a Lipschitz partition of the unity subordinate to the cover made with the $I_i$'s and for any $i$ let $\psi_i\in {\mathcal{F}}_2$ be such that $|\phi_i(t)-\psi_i(t)|<\eps$ for every $t\in[0,1]$. Then we have $W_t:=\sum_i\psi_i(t)W_{i,t}\in \testvfn$ and
\[
\norm{V-W}_t\leq\norm{V-\sum_i\phi_i(t)W_i}_t+\norm{\sum_i(\psi_i(t)-\phi_i(t))W_i}_t \stackrel{\eqref{eq:b1},\eqref{eq:b2}}\leq \eps+\eps( \|V\|_{\mathscr C(\sppi)}+2\eps)
\]
for any $t\in[0,1]$. The conclusion follows by the arbitrariness of $\eps>0$.
\end{proof}
\subsection{The space \texorpdfstring{$\mathscr{L}^2(\ppi)$}{L2(pi)}}

We start defining  the class of Borel vector fields along $\ppi$:
\begin{definition}[Borel vector fields along $\ppi$]\label{def:borelvf}
We say that  $V\in \vf$ is \emph{Borel} provided
\begin{equation}\label{convect_deriv_f9}
[0,1]\ni t\longmapsto\int\la V_t,W_t\ra\,\d\ppi\;\mbox{ is a Borel function,}
\end{equation}
for every $W\in\testvfn$.
\end{definition}
Notice that thanks to Lemma \ref{lem:separability_TestG(X)} this notion would be unaltered if we require \eqref{convect_deriv_f9} to hold for any $W\in\testvf$. Also, Proposition \ref{prop:test_curves_continuous} ensures that  test vector fields are Borel.
%
%
%
%
We have the following basic result:
\begin{proposition}\label{prop:ptwse_norm_Borel}
Let $V\in\vf$ be Borel. Then the map $[\![V]\!]:[0,1]\to[0,\infty)$ is Borel.
\end{proposition}
\begin{proof} From \eqref{eq:densevfn} we deduce that
$${[\![V]\!]}_t^2=\sup_{W\in{\rm Test}\Gamma_\N({\rm X})}
\left(2\int\la V_t,W_t\ra\,\d\ppi-{[\![W]\!]}_t^2\right)
\quad\mbox{ for every }t\in[0,1].$$
and the thesis follows.
\end{proof}
We can now define the space $\mathscr L^2(\ppi)$:
\begin{definition}[The space $\mathscr{L}^2(\ppi)$]
The space $\mathscr L^2(\ppi)$ is the space of all Borel vector fields $V\in\vf$ such that
\[
{\|V\|}^2_{\mathscr{L}^2(\sppi)}:=\int_0^1{[\![V]\!]}_t^2\,\d t=
\int_0^1\!\!\!\int{|V_t|}^2\,\d\ppi\,\d t<+\infty,
\]
where we identify $V,\tilde V\in\vf$ if $V_t=\tilde V_t$ for a.e.\ $t\in[0,1]$.
\end{definition}
Clearly $\big(\mathscr{L}^2(\ppi),{\|\cdot\|}_{\mathscr{L}^2(\sppi)}\big)$ is
a normed space, wherein $\testvf$ is embedded. Adapting the classical arguments concerning the standard $L^2$ spaces we have the following:
\begin{proposition}[Basic properties of  $\mathscr{L}^2(\ppi)$]\label{prop:L2(pi)_Banach}
The space $\mathscr{L}^2(\ppi)$ is a Hilbert space and if $V_n\to V$ in $\mathscr L^2(\ppi)$ then there is a subsequence such that $V_{n,t}\to V_t$ in $\e_t^*L^2(T\X)$ for a.e.\ $t\in[0,1]$.
\end{proposition}
\begin{proof} It is clear that the $\mathscr L^2(\ppi)$ norm comes from the scalar product
\[
\la V,W\ra_{\mathscr L^2(\sppi)}:=\int_0^1\int \la V_t,W_t\ra\,\d\ppi\,\d t.
\]
To conclude the proof we shall show that if $(V_n)$ is a sequence of Borel vector fields in $\mathscr L^2(\ppi)$ such that $\sum_n\|V_{n+1}-V_n\|_{\mathscr L^2(\sppi)}<\infty$, then such sequence has a limit $V\in\mathscr L^2(\ppi)$ and for a.e.\ $t$ it holds $V_{n,t}\to V_t$ in $\e_t^*L^2(T\X)$.

Define the Borel function $g:[0,1]\to[0,+\infty]$ as $g:=\sum_n[\![V_{n+1}-V_n]\!]$ and notice that since 
\[
\Big\|\sum_{n=1}^N[\![V_{n+1}-V_n]\!]\Big\|_{L^2(0,1)}\leq\sum_{n=1}^N\|V_{n+1}-V_n\|_{\mathscr L^2(\sppi)}\leq\sum_{n=1}^\infty\|V_{n+1}-V_n\|_{\mathscr L^2(\sppi)}<\infty\quad\forall N\in\N,
\]
we have that $g\in L^2(0,1)$. Let $N:=\{t\,:\,g(t)=+\infty\}$ and notice that for every $t\in[0,1]\setminus N$ we have
\begin{equation}
\label{eq:boundV}
\sum_{n=1}^\infty\|V_{n+1,t}-V_{n,t}\|_{\e_t^*L^2(T\X)}=\sum_{n=1}^\infty\norm{V_{n+1}-V_n}_t=g(t)<\infty,
\end{equation}
proving that $(V_{n,t})$ is a Cauchy sequence in $\e_t^*L^2(T\X)$. Then define
$$V_t:=\left\{\begin{array}{ll}
\lim_n V_{n,t}\in\e_t^*L^2(T{\rm X})\\
0\in\e_t^*L^2(T{\rm X})
\end{array}\quad\begin{array}{ll}
\text{ if }t\in[0,1]\setminus N.\\
\text{ if }t\in N.
\end{array}\right.
$$
Notice that for every $W\in\testvf$ we have $\int\la V_t,W_t\ra\,\d\ppi=\lim_n\int\la V_{n,t},W_t\ra\,\d\ppi$ for all $t\in[0,1]\setminus N$, hence the map $[0,1]\ni t\mapsto\int\la V_t,W_t\ra\,\d\ppi$ is Borel and, by arbitrariness of $W$, this shows that $V$ is Borel. Since trivially we have $\norm {V}_t\leq \norm {V_1}_t+ \sum_{n=1}^\infty\norm{V_{n+1}-V_n}_t$, by \eqref{eq:boundV} we see that $V\in\mathscr L^2(\ppi)$. Now to check that $V_n\to V$ in $\mathscr L^2(\ppi)$ notice that, again by \eqref{eq:boundV}, the sequence $\norm{V-V_n}_t$ is dominated in $L^2(0,1)$ and that for every $t\in[0,1]\setminus N$ we have
\[
\lim_{n\to\infty}\norm{V-V_n}_t\leq \lim_{n\to\infty} \lim_{m\to\infty} \sum_{i=n}^{m}\norm{V_{i+1}-V_i}_t\stackrel{\eqref{eq:boundV}}=0,
\]
so that the conclusion follows by the dominate convergence theorem.
\end{proof}
\begin{proposition}[Density of $\testvfn$ in $\mathscr{L}^2(\ppi)$]
\label{prop:TestG(X)_dense_in_L2}
The space $\testvfn$ is dense in $\mathscr{L}^2(\ppi)$. In particular, $\mathscr L^2(\ppi)$ is separable.
\end{proposition}
\begin{proof} 
Let $(Z_k)$ be an enumeration of the elements in $\testvfn$, pick a Borel vector field $V\in\mathscr L^2(\ppi)$ and choose $\eps>0$. Then for every $k\in\N$ let $\widetilde{G}_k:=\big\{t\in[0,1]\,:\,{[\![V-Z_k]\!]}_t<\eps\big\}$ and put $G_1:=\tilde G_1$ and
$G_k:=\widetilde{G}_k\setminus(\widetilde{G}_1\cup\ldots\cup\widetilde{G}_{k-1})$ for $k>1$. Then \eqref{eq:densevfn} grants that
${(G_k)}_{k\geq 1}$ is a Borel partition of $[0,1]$.

Now for $m\in\N\cup\{\infty\}$  define $W_m\in\mathscr{L}^2(\ppi)$ as $W_{m,t}:=\sum_{k=1}^m\nchi_{G_k}(t)\,Z_{k,t}$. Observe that
${\|W_\infty-V\|}_{\mathscr{L}^2(\sppi)}<\eps$ by definition of $G_k$.
Moreover, for each $m\geq 1$ one has that
$${\|W_m-W_\infty\|}_{\mathscr{L}^2(\sppi)}^2
=\sum_{k=m+1}^\infty\int_{G_k}{[\![Z_k]\!]}_t^2\,\d t\leq
\int_{\bigcup_{k>m}G_k}2\,\big({[\![V]\!]}^2_t+\eps^2\big)\,\d t,$$
so that accordingly $\lim_{m\to\infty}{\|W_m-W_\infty\|}_{\mathscr{L}^2(\sppi)}=0$ by the dominated convergence theorem.

Hence to conclude it is sufficient to show that each $W_m$ belongs to the $\mathscr L^2(\ppi)$-closure of $\testvfn$ and in turn this will follow if we prove that for $Z\in\testvfn$ and $G\subset[0,1]$ Borel the vector field $\nchi_GZ$ belongs to the $\mathscr L^2(\ppi)$-closure of $\testvfn$. To see this, simply let $(\varphi_n)\subset {\rm LIP}([0,1])$ be uniformly bounded and a.e.\ converging to $\nchi_G$, notice that $\varphi_nZ\in\testvf$ and that an application of the dominate convergence theorem shows that $\varphi_nZ\to \nchi_GZ$ in $\mathscr L^2(\ppi)$.
\end{proof}

Now consider the speed $\ppi'_t$, associated to any test plan $\ppi$
by Theorem \ref{thm:speed_test_plan}. 
\begin{proposition}\label{prop:pi'_in_L2(pi)}
The (equivalence class up to a.e.\ equality of the) map $t\mapsto\ppi'_t$ is an element of the space $\mathscr{L}^2(\ppi)$.
\end{proposition}
\begin{proof}
We have $\ppi'_t\in\e_t^*L^2(T{\rm X})$ for a.e.\ $t\in[0,1]$ and
$$\int_0^1\!\!\int{|\ppi'_t|}^2\,\d\ppi\,\d t\overset{\eqref{eq:speed_test_plan_2}}{=}
\int_0^1\!\!\int|\dot{\gamma}_t|^2\,\d\ppi(\gamma)\,\d t<+\infty$$
by the very definition of test plan. Hence we need only to show that $t\mapsto\ppi'_t$ has a Borel representative in the sense of Definition \ref{def:borelvf}.

Notice that for any $f\in W^{1,2}(\X)$ by Proposition \ref{prop:perborelsp} we have that the map $t\mapsto (\e_t^*\d f)(\ppi'_t)$ has a Borel representative. Hence the same holds for $t\mapsto \psi(t)\nchi_U\la\e_t^*\nabla f,\ppi'_t\ra$ for every $\psi\in{\rm LIP}([0,1])$ and $U\subset\Gamma(\X)$ Borel. Therefore there exists a Borel negligible set $N\subset[0,1]$ such that for every $V\in\testvfn$ the map $t\mapsto \int\la V_t,\ppi'_t\ra\,\d\ppi $, set to 0 on $N$, is Borel. This is sufficient to conclude.
\end{proof}
We conclude the section by pointing out that $\mathscr L^2(\ppi)$ can also be seen as the pullback of $L^2(T\X)$ via the evaluation map $\e:\Gamma(\X)\times[0,1]\to\X$ defined as $\e(\gamma,t):=\gamma_t$. To this aim, let us start by defining the following operations:
\begin{itemize}
\item[$\rm (i)$] Given $f\in L^\infty(\ppi\times\mathcal{L}_1)$ and $V\in\mathscr{L}^2(\ppi)$,
we define $fV\in\mathscr{L}^2(\ppi)$ as
\begin{equation}\label{eq:mult_L2(pi)}
(fV)_t:=f(\cdot,t)V_t\quad\mbox{ for }\mathcal{L}^1\mbox{-a.e.\ }t\in[0,1].
\end{equation}
\item[$\rm (ii)$] To each $V\in\mathscr{L}^2(\ppi)$ we associate the map $|V|\in L^2(\ppi\times\mathcal{L}_1)$,
defined by
\[
|V|(\gamma,t):=|V_t|(\gamma)\quad\mbox{ for }(\ppi\times\mathcal{L}_1)\mbox{-a.e.\ }(\gamma,t)\in\Gamma({\rm X})\times[0,1].
\]
\end{itemize}
It is clear that these operations give  $\mathscr{L}^2(\ppi)$ the structure of an $L^2(\ppi\times\mathcal{L}_1)$-normed module.

We then define the linear continuous operator $\Phi:L^2(T\X)\to \mathscr{L}^2(\ppi)$ by putting 
\[
{\Phi(v)}_t:=\e_t^*v,\qquad\text{for $\mathcal{L}_1$-a.e.\ $t\in[0,1]$.}
\]
We then have:
\begin{proposition}[$\mathscr L^2(\ppi)$ as pullback]\label{prop:l2pb}
We have $\big(\mathscr{L}^2(\ppi),\Phi\big)\sim\big(\e^*L^2(T{\rm X}),\e^*\big)$, i.e.:
\begin{equation}\label{eq:L2(pi)_is_pullback}\begin{split}
\big|\Phi(v)\big|=|v|\circ\e\,&\quad\mbox{ holds }(\ppi\times\mathcal{L}_1)\mbox{-a.e., for any }v\in L^2(T{\rm X}),\\
\big\{\Phi(v)\,:\,v\in L^2(T{\rm X})\big\}&\quad\mbox{ generates }\mathscr{L}^2(\ppi)\mbox{ as a module.}
\end{split}
\end{equation}
\end{proposition}
\begin{proof}
The first in \eqref{eq:L2(pi)_is_pullback} follows by noticing
that $\big|\Phi(v)\big|(\gamma,t)=|\e_t^*v|(\gamma)=\big(|v|\circ\e\big)(\gamma,t)$ holds
for $(\ppi\times\mathcal{L}_1)$-a.e.\ $(\gamma,t)$, the second one stems from the density of $\testvf$ in $\mathscr{L}^2(\ppi)$.
\end{proof}
Notice that the notion of pullback module $\big(\e^*L^2(T{\rm X}),\e^*\big)$ makes no (explicit) reference to the concept of `test vector field' as defined it in Section \ref{se:test}. Thus this last proposition is also telling that the choice of using these test object to check Borel regularity, which a priori might seem arbitrary, leads in fact to a canonical interpretation of $\mathscr L^2(\ppi)$.
\begin{remark}[$\mathscr L^2(\ppi)$ as direct integral]{\rm
The construction of $\mathscr L^2(\ppi)$ can be summarized by saying that such space is the direct integral of the $\e_t^*L^2(T\X)$, the space of Borel vector fields being the so-called `measurable sections' and the set $\testvfn$ being the one used to check measurability.
}\fr\end{remark}
\subsection{The space $\mathscr{C}(\ppi)$}
Here we introduce and briefly study those vector fields in $\vf$ which are `continuous in time'. We start with the following definition:
\begin{definition}[The space $\mathscr{C}(\ppi)$]
Let $V\in\vf$. Then we say that $V$ is a \emph{continuous vector field} provided 
\begin{equation}
\label{eq:contsc}
[0,1]\ni t\quad \mapsto\quad \int\la V_t,W_t\ra\,\d\ppi\quad\text{ is continuous}
\end{equation}
for every $W\in\testvfn$ and 
\begin{equation}
\label{eq:contno}
[0,1]\ni t\quad \mapsto\quad {[\![V]\!]}_t\qquad\text{ is continuous.}
\end{equation}
We denote the family of all  continuous vector fields by $\mathscr{C}(\ppi)$ and, for every $V\in\mathscr C(\ppi)$, we put
\[
\|V\|_{\mathscr C(\ppi)}:=\max_{t\in[0,1]}\norm{V}_t.
\]
\end{definition}
Lemma \ref{lem:separability_TestG(X)} ensures that this definition would be unaltered if we require \eqref{eq:contsc} to hold for any $W\in\testvf$. Also, Proposition \ref{prop:test_curves_continuous} gives that $\testvf\subset \mathscr C(\ppi)$.

It is not obvious that $\mathscr C(\ppi)$ is a vector space, the problem being in checking that \eqref{eq:contno} holds for linear combinations. This will be a consequence of the density of $\testvfn$ in $\mathscr C(\ppi)$, which is part of the content of the next result:
\begin{proposition}   $(\mathscr{C}(\ppi),\|\cdot\|_{\mathscr C(\sppi)})$  is a  separable Banach space, with $\testvfn$ being dense.
\end{proposition}
\begin{proof}
Let $V_1,V_2\in\mathscr C(\ppi)$ and notice that using \eqref{eq:contsc}, \eqref{eq:contno} and arguing exactly as in the proof of Lemma \ref{lem:separability_TestG(X)} we can find $(W_{1,n}),(W_{2,n})\subset \testvfn$ such that the functions $t\mapsto\norm{V_i-W_{i,n}}_t$ uniformly converge to 0 as $n\to\infty$, $i=1,2$.

Now observe that since $W_{1,n}+W_{2,n}\in\testvfn$ the map $t\mapsto\norm{W_{1,n}+W_{2,n}}_t$ is continuous and that for every $t\in[0,1]$ we have
\[
\big|\norm{V_1+V_2}_t-\norm{W_{1,n}+W_{2,n}}_t\big|\leq \norm{V_1-W_{1,n}+V_2-W_{2,n}}_t\leq \norm{V_1-W_{1,n}}_t+\norm{V_2-W_{2,n}}_t.
\]
Hence $t\mapsto \norm{V_1+V_2}_t$ is the uniform limit of continuous functions and thus continuous itself. Since trivially $\mathscr C(\ppi)$ is closed by multiplication by scalars we proved that it is a vector space. That $\|\cdot\|_{\mathscr C(\ppi)}$ is a complete norm on it is trivial and the density of $\testvfn$ has already been shown, hence the proof is finished.
\end{proof}
A useful consequence of the density of test vector fields is the following strengthening of the continuity property:
\begin{corollary}\label{cor:contnorm}
Let $V\in \mathscr C(\ppi)$. Then the map $t\mapsto |V_t|^2\in L^1(\ppi)$ is continuous.
\end{corollary}
\begin{proof}
For $V\in\testvf$ the claim has been proved in Proposition \ref{prop:test_curves_continuous}. Now notice that for $V,W\in\mathscr C(\ppi)$ we have
\[
\int||V_t|^2-|W_t|^2|\,\d\ppi\leq \int|V_t+W_t|\,|V_t-W_t|\,\d\ppi\leq (\|V\|_{\mathscr C(\sppi)}+\|W\|_{\mathscr C(\sppi)})\norm{V-W}_t
\]
thus showing that if $V_n\to V$ in $\mathscr C(\ppi)$ then $t\mapsto |V_{n,t}|^2\in L^1(\ppi)$ uniformly converge to $t\mapsto |V_t|^2\in L^1(\ppi)$. The conclusion then follows from the density of $\testvf$ in $\mathscr C(\ppi)$.
\end{proof}

\subsection{The spaces \texorpdfstring{$\mathscr{W}^{1,2}(\ppi)$}{W(pi)} and $\mathscr H^{1,2}(\ppi)$}

Throughout all this section we shall further assume that the test plan $\ppi$ is Lipschitz in the sense of Definition \ref{def:lip_test_plan}.

\bigskip

Let $v\in W^{1,2}_C(T\X)$ and notice that the map from $L^0(T\X)$ to $\e_t^*L^0(T\X)$ defined by
\[
w\quad\mapsto\quad \e_t^*(\nabla_wv)
\]
satisfies
\[
|\e_t^*(\nabla_wv)|\leq |\nabla v|_{\sf HS}\circ\e_t\,|w|\circ\e_t\qquad\ppi-a.e..
\]
Hence by the universal property of the pullback given in Proposition \ref{prop:pullback_L0_univ_prop} we know that there exists a unique $L^0(\ppi)$-linear continuous operator, which we shall call ${\rm Cov}(v,\cdot)$ from $\e_t^*L^0(T\X)$ to $\e_t^*L^0(T\X)$ such that
\[
{\rm Cov}(v,\e_t^*w)=\e_t^*(\nabla_wv)\qquad\forall w\in L^0(T\X)
\]
and such operator satisfies the bound
\begin{equation}
\label{eq:b3}
|{\rm Cov}(v,W)|\leq |\nabla v|_{\sf HS}\circ\e_t|W|\qquad\ppi-a.e..
\end{equation}
We shall be interested in such covariant differentiation along the speed of our test plan: for every $t\in[0,1]$ such that $\ppi'_t$ exists we define the map ${\rm Cov}_t:W^{1,2}_C(T\X)\to \e^*_tL^0(T\X)$ as
\[
{{\rm Cov}_t(v)}:={\rm Cov}(v,\ppi'_t).
\]
Notice the following simple proposition:
\begin{proposition}\label{prop:basecov}
For every $t\in[0,1]$ such that $\ppi'_t$ exists, the map ${{\rm Cov}_t}$ is linear and continuous from $W^{1,2}_C(T\X)$ to $\e^*_tL^2(T\X)$.

Moreover, for every $v\in W^{1,2}_C(T\X)$ the (equivalence class up to a.e.\ equality of the) a.e.\ defined map $t\mapsto {{\rm Cov}_t(v)}\in \e^*_tL^2(T\X)$ is  an element of $\mathscr L^2(\ppi)$.
\end{proposition}
\begin{proof}
The continuity of ${\rm Cov}_t$ as map from $W^{1,2}_C(T\X)$ to $\e^*_tL^2(T\X)$ is a direct consequence of the bounds \eqref{eq:b3} and our assumption \eqref{eq:lip_test_plan}:
\[
\|{{\rm Cov}_t(v)}\|_{\e^*_tL^2(T\X)}^2=\int |{{\rm Cov}_t(v)}|^2\,\d\ppi\stackrel{\eqref{eq:b3}}\leq \int |\nabla v|^2_{\sf HS}\circ\e_t|\ppi'_t|^2\,\d\ppi\leq \sfC(\ppi)\,{\sfL(\ppi)}^2\,{\|v\|}_{W^{1,2}_C(T{\rm X})}^2.
\]
Thanks to this bound, to conclude it is sufficient to show that for any $v\in W^{1,2}_C(T\X)$ the map $t\mapsto {\rm Cov}_t(v)={\rm Cov}(v,\ppi'_t)$ is a.e.\ equal to a Borel element of  $\vf$. Taking into account that $t\mapsto \ppi'_t\in\mathscr L^2(\ppi)$ by Proposition \ref{prop:pi'_in_L2(pi)}, that $\testvf$ is dense in $\mathscr L^2(\ppi)$, the second claim in Proposition \ref{prop:L2(pi)_Banach} and the bound \eqref{eq:b3}, we see that to conclude it is sufficient to show that $t\mapsto {\rm Cov}(v,V_t)$ is a Borel vector field in $\vf$ for any $V\in\testvf$.

Thus fix such $V$, say $V_t=\sum_i\phi_i(t)\nchi_{A_i}\e_t^*v_i$, and let $W_t=\sum_j\psi_j(t)\nchi_{B_j}\e_t^*w_j\in\testvf$ be arbitrary. Notice that since $|v_i|,|w_j|\in L^2\cap L^\infty(\X)$, we have that $\la v_i,\nabla_{w_j}v\ra\in L^1(\X)$ and thus by Theorem \ref{thm:cont_f_circ_e_t} we deduce that $t\mapsto \la v_i,\nabla_{w_j}v\ra\circ\e_t\in L^1(\ppi)$ is continuous for every $i,j$. Therefore
\[
t\qquad\mapsto\qquad\int\big\langle V_t,{\rm Cov}_t(v,W_t)\big\rangle\,\d\ppi
=\sum_{i,j}\varphi_i(t)\,\psi_j(t)\int\nchi_{A_i\cap B_j}\la v_i,\nabla_{w_j}v\ra\circ\e_t
\,\d\ppi
\]
is continuous, thus establishing, by the arbitrariness of $W$, the Borel regularity of $t\mapsto {\rm Cov}(v,V_t)$.
\end{proof}
The `compatibility with the metric' of the covariant derivative yields the following simple but crucial lemma:
\begin{lemma}\label{lem:lebniz_cov}
Let $v,w\in {\rm TestV}(\X)$. Then the map 
$t\mapsto\la v,w\ra\circ\e_t\in L^2(\ppi)$, which is Lipschitz
by Proposition \ref{prop:pi'_t_in_L2}, satisfies
\begin{equation}\label{eq:leibniz_cov_2}
L^2(\ppi)\text{-}\frac{\d}{\d t}\,\la v,w\ra\circ\e_t
=\big\langle{{\rm Cov}_\sppi(v)}_t,\e_t^*w\big\rangle+\big\langle\e_t^*v,{{\rm Cov}_\sppi(w)}_t\big\rangle
\quad\mbox{ for }\mathcal{L}^1\mbox{-a.e.\ }t\in[0,1].
\end{equation}
 \end{lemma}
\begin{proof}
Recall from \cite{Gigli14} that it holds
\[
\d(\la v,w\ra)(z)=\la \nabla_zv,w\ra+\la v,\nabla_zw\ra \quad\mm-a.e.\qquad\forall z\in L^0(T\X)
\]
and notice that from the defining property of pointwise norm in the pullback and by polarization we obtain that $\la\e_t^*v_1,\e_t^*v_2\ra=\la v_1,v_2\ra\circ\e_t$ for every $v_1,v_2\in L^0(T\X)$. Thus we have that the identity
\begin{equation}\label{eq:leibniz_cov_1}
\big(\e_t^*\d\la v,w\ra\big)(Z)=
\big\langle{\rm Cov}_t(v,Z),\e_t^*w\big\rangle+\big\langle\e_t^*v,{\rm Cov}_t(w,Z)\big\rangle
\end{equation} 
holds for every $Z\in\e_t^*L^2(T{\rm X})$ of the form $Z_t=\e_t^*z$ for some $z\in L^2(T\X)$. Since both sides of this identity are  $L^\infty(\ppi)$-linear and continuous in $Z$, we see that \eqref{eq:leibniz_cov_1} holds for generic $Z\in\e_t^*L^2(T{\rm X})$. The conclusion comes picking $Z=\ppi'_t$ and recalling Proposition \ref{prop:pi'_t_in_L2}.
\end{proof}

We now want to introduce a new differential operator, initially defined only on $\testvf$ and then extended to more general vector fields. To this aim the following lemma will be useful.
\begin{lemma}\label{le:perconv}
Let $(\varphi_i),(\psi_j)\subset {\rm LIP}([0,1])$, $(A_i),(B_j)$ Borel partitions of $\Gamma(\X)$ and $(v_i),(w_j)\subset{\rm TestV}(\X)$, where $i=1,\ldots,n$ and $j=1,\ldots,m$. Assume that
\begin{equation}
\label{eq:assder}
\sum_i\nchi_{A_i}\varphi_i(t)\,\e_t^*v_i=\sum_j\nchi_{B_j}\psi_j(t)\,\e_t^*w_j\qquad\text{ for every }t\in[0,1].
\end{equation}
Then for a.e.\ $t$ it holds:
\begin{equation}
\label{eq:perconv}
\begin{split}
\sum_i\nchi_{A_i}\varphi'_i(t)\,\e_t^*v_i&=\sum_j\nchi_{B_j}\psi_j'(t)\,\e_t^*w_j,\\
\sum_i\nchi_{A_i}\varphi_i(t)\,{\rm Cov}_t(v_i)&=\sum_j\nchi_{B_j}\psi_j(t)\,{\rm Cov}_t(w_j).
\end{split}
\end{equation}
\end{lemma}
\begin{proof}
For the first in \eqref{eq:perconv} we notice that our assumption \eqref{eq:assder} and Proposition \ref{prop:l2pb} yield that $\sum_i\nchi_{A_i\times[0,1]}(\cdot,t)\,\varphi_i(t)\,\e^*v_i=\sum_j\nchi_{B_j\times[0,1]}(\cdot,t)\,\psi_j(t)\,\e^*w_j$ as elements in $\e^*L^2(T\X)\sim\mathscr L^2(\ppi)$, thus we can differentiate in time and conclude using again Proposition \ref{prop:l2pb}. 

For the second in \eqref{eq:perconv} start noticing that our assumption \eqref{eq:assder} and the very definition of pullback imply that for any $i,j$ and every $t\in[0,1]$ it holds $\nchi_C\varphi_i(t)\,v_i=\nchi_C\psi_j(t)\,w_j$,
where $C:=\big\{\frac{\d(\e_t)_*(\nchi_{A_i\cap B_j}\sppi)}{\d(\e_t)_*\sppi}>0\big\}$. This identity and the locality of the covariant derivative give that $\nchi_C\varphi_i(t)\,\nabla_zv_i=\nchi_C\psi_j(t)\,\nabla_zw_j$ for every $z\in L^2(T\X)$. Applying the pullback map on both sides and noticing that $\nchi_C\circ\e_t\geq \nchi_{A_i\cap B_j}$ we deduce that
\[
\nchi_{A_i\cap B_j}\varphi_i(t)\,{\rm Cov}(v_i,Z)=\nchi_{A_i\cap B_j}\psi_j(t)\,{\rm Cov}(w_j,Z)
\]
for every $Z$ of the form $Z_t=\e_t^*z$. From the $L^\infty(\ppi)$-linearity in $Z$ of both sides and the arbitrariness of $i,j$ the conclusion follows.
\end{proof}
We can now define the convective derivative of test vector fields:
\begin{definition}[Convective derivative along a test plan]\label{def:convect}
We define the \emph{convective derivative} operator $\widetilde{\rm D}_\sppi:\testvf \to
\mathscr{L}^2(\ppi)$ as follows:
to the element $V\in\testvf$, of the form $V_t=\sum_{i=1}^n\varphi_i(t)\,\nchi_{A_i}\,\e_t^*v_i$,
we associate the vector field $\widetilde{\rm D}_\sppi V\in\mathscr{L}^2(\ppi)$ given by
\begin{equation}\label{eq:convect}
{(\widetilde{\rm D}_\sppi V)}_t:=\sum_{i=1}^n\nchi_{A_i}\Big(\varphi'_i(t)\,\e_t^*v_i
+\varphi_i(t)\,{{\rm Cov}_t(v_i)}\Big)\quad\mbox{ for }\mathcal{L}^1\mbox{-a.e.\ }t\in[0,1].
\end{equation}
For the sake of simplicity, we will briefly write $\widetilde{\rm D}_\sppi V_t$
instead of ${(\widetilde{\rm D}_\sppi V)}_t$.
\end{definition}
Notice that Lemma \ref{le:perconv} ensures that the right hand side of \eqref{eq:convect} depends only on $V$ and not on the way we write it as $V_t=\sum_{i=1}^n\varphi_i(t)\,\nchi_{A_i}\,\e_t^*v_i$. The fact that the right hand side of \eqref{eq:convect} defines a Borel vector field in $\vf$ follows directly from Proposition \ref{prop:basecov}; to see that it belongs to $\mathscr L^2(\ppi)$ notice that $(t\mapsto \e_t^*v_i),{\rm Cov}_\sppi(v_i)\in\mathscr L^2(\ppi)$ for every $i$ and  that the $\varphi_i$'s are Lipschitz. 

Therefore the definition is well posed and is then clear  that $\widetilde{\rm D}_\sppi$ is a linear operator.
\bigskip

The convective derivative has the following simple and crucial property, which is a direct consequence of Lemma \ref{lem:lebniz_cov}.
\begin{proposition}\label{prop:leibniz_convect_integral}
Let $V,W\in\testvf$. Then the map $t\mapsto\la V_t,W_t\ra\in L^2(\ppi)$
is Lipschitz and satisfies
\begin{equation}\label{eq:leibniz_convect_integral_1}
L^2(\ppi)\text{-}\frac{\d}{\d t}\,\la V_t,W_t\ra
=\la\widetilde{\rm D}_\sppi V_t,W_t\ra+
\la V_t,\widetilde{\rm D}_\sppi W_t\ra
\quad\mbox{ for }\mathcal{L}^1\mbox{-a.e.\ }t\in[0,1].
\end{equation}
\end{proposition}
\begin{proof}
By bilinearity, to prove \eqref{eq:leibniz_convect_integral_1}  is sufficient to consider the case $V_t=\varphi(t)\nchi_A\,\e_t^*v$ and 
$W_t=\psi(t)\,\nchi_{B}\,\e_t^*w$ for $v,w\in {\rm TestV}({\rm X})$. Lemma  \ref{lem:lebniz_cov} ensures that $t\mapsto \la\e_t^*v,\e_t^*w\ra=\la v,w\ra\circ\e_t\in L^2(\ppi)$ is Lipschitz and it is then clear that $t\mapsto \la V_t,W_t\ra=\nchi_{A\cap B}\varphi(t)\psi(t)\la\e_t^*v,\e_t^*w\ra$ is also Lipschitz. The identity \eqref{eq:leibniz_convect_integral_1} now follows from \eqref{eq:leibniz_cov_2} and the Leibniz rule.
%
\end{proof}
This last proposition allows to `integrate by parts' and extend the definition of convective derivative to `Sobolev vector fields along $\ppi$'.

Let us define the  \emph{support} ${\rm spt}(V)$ of a test vector field $V\in{\rm TestVF}(\ppi)$ as the closure of the set of $t$'s for which $V_t\neq 0$ and let us introduce the space of sections with compact support in $(0,1)$:
\[
\testvfc:=\big\{V\in\testvf\,:\,{\rm spt}(V)\subseteq(0,1)\big\}.
\]
A simple cut-off argument shows that $\testvfc$  is $\mathscr{L}^2(\ppi)$-dense in $\testvf$ and hence in $\mathscr L^2(\ppi)$.
\begin{definition}[The space \texorpdfstring{$\mathscr{W}^{1,2}(\boldsymbol{\pi})$}{W^1,2(pi)}]\label{def:w12}
The \emph{Sobolev space} $\mathscr{W}^{1,2}(\ppi)$ is the vector subspace of $\mathscr{L}^2(\ppi)$
consisting of all those $V\in\mathscr{L}^2(\ppi)$ such that there exists $Z\in\mathscr{L}^2(\ppi)$ satisfying
\begin{equation}\label{eq:W^1,2(pi)_1}
\int_0^1\!\!\!\int\la V_t,\widetilde{\rm D}_\sppi W_t\ra\,\d\ppi\,\d t=
-\int_0^1\!\!\!\int\la Z_t,W_t\ra\,\d\ppi\,\d t\quad\mbox{ for every }W\in{\rm TestVF}_c(\ppi).
\end{equation}
In this case the section $Z$, whose uniqueness is granted by density of $\testvfc$ in
$\mathscr{L}^2(\ppi)$, can be unambiguously denoted by ${\rm D}_\sppi V$ and called
\emph{convective derivative} of $V$.
We endow $\mathscr{W}^{1,2}(\ppi)$ with the norm ${\|\cdot\|}_{\mathscr{W}^{1,2}(\sppi)}$, defined by
\[
{\|V\|}_{\mathscr{W}^{1,2}(\sppi)}:=\sqrt{{\|V\|}_{\mathscr{L}^2(\sppi)}^2
+{\|{\rm D}_\sppi V\|}_{\mathscr{L}^2(\sppi)}^2}
\quad\mbox{ for every }V\in\mathscr{W}^{1,2}(\ppi).
\]
\end{definition}

This choice of terminology is consistent with that of Definition \ref{def:convect}:
\begin{proposition}\label{prop:consistency}
Let $V\in\testvf$. Then $V\in\mathscr{W}^{1,2}(\ppi)$ and 
${\rm D}_\sppi V=\widetilde{\rm D}_\sppi V$.
\end{proposition}
\begin{proof}
Fix $W\in\testvfc$. We know from Proposition \ref{prop:leibniz_convect_integral} that
$[0,1]\ni t\mapsto\int\la V_t,W_t\ra\,\d\ppi$ is an absolutely continuous function, so that
\eqref{eq:leibniz_convect_integral_1} gives, after integration, that
$$0=\int\la V_1,W_1\ra\,\d\ppi-\int\la V_0,W_0\ra\,\d\ppi
=\int_0^1\!\!\!\int\la\widetilde{\rm D}_\sppi V_t,W_t\ra\,\d\ppi\,\d t+
\int_0^1\!\!\!\int\la V_t,\widetilde{\rm D}_\sppi W_t\ra\,\d\ppi\,\d t.$$
This proves that $V$ satisfies \eqref{eq:W^1,2(pi)_1} with $Z=\widetilde{\rm D}_\sppi V$.
\end{proof}
\begin{proposition}[Basic properties of \texorpdfstring{$\mathscr{W}^{1,2}(\ppi)$}{W_1,2(pi)}]
\label{prop:W12(ppi)_Hilbert}
The following hold:
\begin{itemize}
\item[i)] ${\rm D}_\sppi$ is a closed operator from $\mathscr{L}^2(\ppi)$ into itself, i.e.\ its graph is closed in the product space $\mathscr{L}^2(\ppi)\times\mathscr{L}^2(\ppi)$.
\item[ii)] $\mathscr{W}^{1,2}(\ppi)$ is a separable Hilbert space.
\item[iii)] Let $V,Z\in\mathscr L^2(\ppi)$. Then $V\in \mathscr{W}^{1,2}(\ppi)$ and $Z={\rm D}_\sppi V$ if and only if for every $W\in \testvf$ the map $t\mapsto\la V_t,W_t\ra$ belongs to $W^{1,1}([0,1],L^1(\ppi))$ with derivative given by
\begin{equation}
\label{eq:b4}
\frac{\d}{\d t} \la V_t,W_t\ra=\la V_t,{\rm D}_\sppi W_t\ra+\la Z_t,W_t\ra \qquad a.e.\ t.
\end{equation}
\end{itemize}
\end{proposition}
\begin{proof}

\noindent{\bf (i)} 
Let ${(V_n)}\subseteq\mathscr{W}^{1,2}(\ppi)$ be a sequence such that $V_n\to V$ and ${\rm D}_\sppi V_n\to Z$ in $\mathscr{L}^2(\ppi)$ for some $V,Z\in \mathscr{L}^2(\ppi)$. Then for arbitrary  $W\in{\rm TestVF}_c(\ppi)$ we have
\[\begin{split}
\int_0^1\!\!\!\int\la V_t,{\rm D}_\sppi W_t\ra\,\d\ppi\,\d t
&=\lim_{n\to\infty}\int_0^1\!\!\!\int\la V^n_t,{\rm D}_\sppi W_t\ra\,\d\ppi\,\d t\\
&=-\lim_{n\to\infty}\int_0^1\!\!\!\int\la{\rm D}_\sppi V^n_t,W_t\ra\,\d\ppi\,\d t=-\int_0^1\!\!\!\int\la Z_t,W_t\ra\,\d\ppi\,\d t,
\end{split}\]
proving that  $V\in \mathscr W^{1,2}(\ppi)$ with ${\rm D}_\sppi V=Z$, which was the claim.

\noindent{\bf (ii)} Consequence of what just proved and the fact that the map
$$\mathscr{W}^{1,2}(\ppi)\ni V\quad\mapsto\quad(V,{\rm D}_\sppi V)\in\mathscr{L}^2(\ppi)\times\mathscr{L}^2(\ppi)$$
is an isometry, provided we endow $\mathscr{L}^2(\ppi)\times\mathscr{L}^2(\ppi)$,
with the (separable, by Proposition \ref{prop:TestG(X)_dense_in_L2}) norm ${\big\|(V,Z)\big\|}^2:={\|V\|}_{\mathscr{L}^2(\sppi)}^2
+{\|Z\|}_{\mathscr{L}^2(\sppi)}^2$. 

\noindent{\bf (iii)} The `if' trivially follows from \eqref{eq:b4} by integration. For the `only if', fix $W\in \testvf$ and let $\varphi\in C^1_c(0,1)$ and $\Gamma\subset\Gamma(\X)$ Borel. Then $t\mapsto \varphi(t)\nchi_\Gamma W_t$ is in $\testvfc$ and a direct computation shows that ${\rm D}_\sppi(\varphi \nchi_\Gamma W)_t=\varphi'(t)\nchi_\Gamma W_t+\varphi(t)\nchi_\Gamma {\rm D}_\sppi W_t$. Hence writing the defining property \eqref{eq:W^1,2(pi)_1} with $\varphi \nchi_\Gamma W$ in place of $W$ we obtain, after rearrangement, that
\[
\int_0^1\varphi'(t)\int_\Gamma \la V_t,W_t\ra\,\d\ppi\,\d t=-\int_0^1\varphi(t)\int_\Gamma\la V_t,{\rm D}_\sppi W_t\ra+\la Z_t,W_t\ra\,\d\ppi\,\d t.
\]
The arbitrariness of $\varphi,\Gamma$ and Proposition \ref{prop:carw1p} yield the claim.
\end{proof}

We just proved that $\testvf$ is contained in $\mathscr W^{1,2}(\ppi)$, but we don't know if it is dense. Hence the following definition is meaningful:
\begin{definition}[The space \texorpdfstring{$\mathscr{H}^{1,2}(\ppi)$}{H^1,2(pi)}]\label{def:h12}
 $\mathscr{H}^{1,2}(\ppi)$ is the $\mathscr{W}^{1,2}(\ppi)$-closure of $\testvf$.
\end{definition}
Clearly, $\mathscr{H}^{1,2}(\ppi)$ is a separable Hilbert space. A key feature of elements of  $\mathscr{H}^{1,2}(\ppi)$ is that they admit a continuous representative (much like Sobolev functions on the interval):
\begin{theorem}\label{thm:H12(ppi)_in_C(ppi)}
The inclusion $\testvf\hookrightarrow\mathscr{C}(\ppi)$ uniquely extends to a linear continuous and injective operator  $\iota:\mathscr{H}^{1,2}(\ppi)\to \mathscr{C}(\ppi)$.
\end{theorem}
\begin{proof}
We claim that
\begin{equation}
\label{eq:b5}
{\|V\|}_{\mathscr{C}(\sppi)}\leq\sqrt{2}\,{\|V\|}_{\mathscr{W}^{1,2}(\sppi)}\qquad\forall V\in\testvf.
\end{equation}
By the density of $\testvf$ in $\mathscr{H}^{1,2}(\ppi)$ this will be enough to obtain the existence of $\iota$. Thus let $V\in
\testvf$, pick $W=V$ in \eqref{eq:leibniz_convect_integral_1} and integrate in $[t_1,t_2]\subset[0,1]$ and w.r.t.\ $\ppi$ to obtain
\[
\begin{split}\big|{[\![V]\!]}_{t_2}^2-{[\![V]\!]}_{t_1}^2\big|&=2\Big|\int_{t_1}^{t_2}\!\!\!\int\la V_t,{\rm D}_\sppi V_t\ra\,\d\ppi\,\d t \Big|\\
&\leq
2\int_{t_1}^{t_2}\!\!\!\int\big|V_t\big|\,\big|{\rm D}_\sppi V_t\big|\,\d\ppi\,\d t
\leq{\|V\|}_{\mathscr{L}^2(\sppi)}^2+{\|{\rm D}_\sppi V\|}_{\mathscr{L}^2(\sppi)}^2.
\end{split}\]
Hence for any $t\in[0,1]$ one has
$$
{[\![V]\!]}_t^2=\int_0^1{[\![V]\!]}_t^2\,\d s
\leq\int_0^1\big|{[\![V]\!]}_t^2-{[\![V]\!]}_s^2\big|\,\d s+{\|V\|}_{\mathscr{L}^2(\sppi)}^2
\leq 2\,{\|V\|}_{\mathscr{W}^{1,2}(\sppi)}^2,
$$
which is our claim \eqref{eq:b5}.

To prove injectivity, let $V\in\mathscr{H}^{1,2}(\ppi)$ be such that $\iota(V)=0$.
Choose a sequence ${(V_n)}\subseteq\testvf$ which is $\mathscr W^{1,2}(\ppi)$-converging to $V$ and notice that, up to pass to a subsequence and using Proposition \ref{prop:L2(pi)_Banach}, we can assume that $V^n_t\to V_t$ for $\mathcal{L}^1$-a.e.\ $t\in[0,1]$. By continuity of the operator $\iota$, one also has
${\|V^n\|}_{\mathscr{C}(\sppi)}={\big\|\iota(V^n)-\iota(V)\big\|}_{\mathscr{C}(\sppi)}\to 0$ and thus
in particular $V^n_t\to 0$ for all $t\in[0,1]$. Therefore $V_t=0$ for $\mathcal{L}^1$-a.e.\ $t\in[0,1]$,
yielding the required injectivity of $\iota$.
\end{proof}

Whenever we will consider an element $V$ of $\mathscr{H}^{1,2}(\ppi)$, we will always implicitly
refer to its unique continuous representative $\iota(V)\in\mathscr{C}(\ppi)$.

Among the several properties of the test sections that can be carried over to the
elements of $\mathscr{H}^{1,2}(\ppi)$, the most important one is the Leibniz
formula for the convective derivatives:
\begin{proposition}[Leibniz formula for ${\rm D}_\sppi$]\label{prop:Leibniz_in_H}
Let $V\in\mathscr{W}^{1,2}(\ppi)$ and $W\in\mathscr{H}^{1,2}(\ppi)$. Then the function $t\mapsto \la V_t,W_t\ra$ is in $W^{1,1}([0,1],L^1(\ppi))$ and its derivative is given by
\[
\frac{\d}{\d t} \la V_t,W_t\ra 
=\la{\rm D}_\sppi V_t,W_t\ra+
\la V_t,{\rm D}_\sppi W_t\ra
\quad\mbox{ for }\mathcal{L}^1\mbox{-a.e.\ }t\in[0,1].
\]
\end{proposition}
\begin{proof}
For $W\in \testvf$ the claim is a direct consequence of point $(iii)$ in Proposition \ref{prop:W12(ppi)_Hilbert}. The general case can be achieved by approximation noticing that the simple inequalities
\[
\begin{split}
\big\|\la V_t,W_t\ra\big\|_{L^1(\sppi\times\mathcal L_1)}&\leq \|V\|_{\mathscr L^2(\sppi)}\|W\|_{\mathscr L^2(\sppi)},\\
\big\|\la{\rm D}_\sppi V_t,W_t\ra+\la V_t,{\rm D}_\sppi W_t\ra\big\|_{L^1(\sppi\times\mathcal L_1)}&\leq 2 \|V\|_{\mathscr W^{1,2}(\sppi)}\|W\|_{\mathscr W^{1,2}(\sppi)},
\end{split}
\]
allow to pass to the limit in the distributional formulation of $\frac\d{\d t} \la V_t,W_t\ra$ as $W$ varies in $\mathscr H^{1,2}(\ppi)$.
\end{proof}
In the next proposition we collect some examples of elements of $\mathscr H^{1,2}(\ppi)$:
\begin{proposition}\label{prop:exh}
Let $\ppi$ be a Lipschitz test plan. Then:
\begin{itemize}
\item[i)] For every $w\in H^{1,2}_C(T\X)$ the vector field $t\mapsto W_t:=\e_t^*w$ belongs to $\mathscr H^{1,2}(\ppi)$ and 
\begin{equation}
\label{eq:derbas}
{\rm D}_\sppi W_t={\rm Cov}_t(w)\qquad  a.e.\ t.
\end{equation}
\item[ii)] Let $W\in\mathscr H^{1,2}(\ppi)$ be such that $|W|,|{\rm D}_\sppi W|\in L^\infty(\ppi\times\mathcal L_1)$ and $a\in W^{1,2}([0,1],L^2(\ppi))$. Then $aW\in\mathscr H^{1,2}(\ppi)$ with 
\begin{equation}
\label{eq:leiba}
{\rm D}_\sppi(aW)_t=a_t'W_t+a_t{\rm D}_\sppi W_t\qquad  a.e.\ t.
\end{equation}
Moreover, if $W\in\mathscr C(\ppi)$ and $a\in AC^2([0,1],L^2(\ppi))$, then $aW\in\mathscr C(\ppi)$.
\end{itemize}
\end{proposition}
\begin{proof}

\noindent{\bf (i)} If $w\in {\rm TestV}(\X)$ we have that $W\in \testvf$ by definition and in this case formula \eqref{eq:derbas}  holds by the definition \eqref{eq:convect} and Proposition \ref{prop:consistency}. The general case can then be obtained by approximating $w$ with vector fields in ${\rm TestV}(\X)$ w.r.t.\ the $ W_C^{1,2}(T\X)$ topology,  using the bounds
\[
\begin{split}
\int_0^1\!\!\!\int{\big|\e_t^*(v)\big|}^2\,\d\ppi\,\d t
&\overset{\phantom{\eqref{eq:b3}}}{=}\int_0^1\!\!\!\int{|v|}^2\circ\e_t\,\d\ppi\,\d t
\leq\sfC(\ppi)\,{\|v\|}^2_{W^{1,2}_C(T{\rm X})},\\
\int_0^1\!\!\!\int{\big|{{\rm Cov}_\sppi(v)}_t\big|}^2\,\d\ppi\,\d t
&\overset{\eqref{eq:b3}}{\leq}
\sfC(\ppi)\,\sfL(\ppi)^2\,{\|v\|}^2_{W^{1,2}_C(T{\rm X})}
\end{split}
\]
and  recalling the closure of ${\rm D}_\sppi$.

\noindent{\bf (ii)} The claim about continuity is obvious, so we concentrate on the other one. Assume at first that $a$ belongs to the space $\mathcal A$ defined as
\[
\mathcal A:=\Big\{\sum_{i=1}^n\varphi_i\nchi_{E_i}\,:\,n\in\N, \ \varphi_i\in{\rm LIP}([0,1]),\ (E_i) \text{ Borel partition of }\Gamma(\X) \Big\}
\]
and that $W\in\testvf$. In this case $aW$ belongs to $\testvf$ as well and formula \eqref{eq:leiba} is a direct consequence of the definitions. Then using the trivial bounds
\[
\begin{split}
\|aW\|_{\mathscr L^2(\sppi)}&\leq \|a\|_{L^\infty(\sppi\times\mathcal L_1)}\|W\|_{\mathscr L^2(\sppi)},\\
\|a'W+a{\rm D}_\sppi W\|_{\mathscr L^2(\sppi)}&\leq\big(\|a\|_{L^\infty(\sppi\times\mathcal L_1)}+\|a'\|_{L^\infty(\sppi\times\mathcal L_1)}\big)\|W\|_{\mathscr W^{1,2}(\sppi)},
\end{split}
\]
the $\mathscr W^{1,2}(\ppi)$-density of $\testvf$ in $\mathscr H^{1,2}(\ppi)$ and the closure of ${\rm D}_\sppi$, we conclude that $aW\in\mathscr H^{1,2}(\ppi)$ for every $a\in \mathcal A$ and $W\in\mathscr H^{1,2}(\ppi)$ and that \eqref{eq:leiba} holds in this case.

Now let $W$ be as in the assumptions and notice that we also have the bounds
\[
\begin{split}
\|aW\|_{\mathscr L^2(\sppi)}&\leq \|a\|_{L^2([0,1],L^2(\sppi))}\||W|\|_{L^\infty(\sppi\times \mathcal L_1)},\\
\|a'W+a{\rm D}_\sppi W\|_{\mathscr L^2(\sppi)}&\leq \|a\|_{W^{1,2}([0,1],L^2(\sppi))}\big(\||W|\|_{L^\infty(\sppi\times \mathcal L_1)}+\||{\rm D}_\sppi W|\|_{L^\infty(\sppi\times \mathcal L_1)}\big),
\end{split}
\]
therefore using again the closure of ${\rm D}_\sppi$, to conclude it is sufficient to prove that $\mathcal A$ is dense in $W^{1,2}([0,1],L^2(\ppi))$. To this aim we argue as follows: for every $n\in\N$ let $(E^n_i)_{i\in\N}$ be a Borel partition of $\supp(\ppi)\subset \Gamma(\X)$ made of sets with positive $\ppi$-measure and diameter $\leq \frac1n$. Then for every $n,N\in\N$ let $P_n^N:L^2(\ppi)\to L^2(\ppi)$ be defined by
\[
P_n^N(f):=\sum_{i=1}^N\nchi_{E^n_i}\frac1{\ppi(E^n_i)}\int_{E^n_i}f\,\d\ppi.
\]
It is clear that $P_n^N$ has operator norm $\leq 1$ for every $n,N\in\N$ and an application of the dominated convergence theorem shows that 
\begin{equation}
\label{eq:limnN}
\lim_n\lim_N P_n^N(f)=f
\end{equation}
for every $f\in C_b\big(\Gamma(\X)\big)$, the limits being intended in $L^2(\ppi)$. Therefore  \eqref{eq:limnN} also holds for every $f\in L^2(\ppi)$. The linearity and continuity of $P_n^N$ also grants that if $t\mapsto a_t$ belongs to $W^{1,2}([0,1],L^2(\ppi))$, then also $t\mapsto P_n^N(a)_t:=P_n^N(a_t)$ belongs to  $W^{1,2}([0,1],L^2(\ppi))$ with 
\begin{equation}
\label{eq:pmder}
\big(P_n^N(a)\big)'_t=P_n^N(a'_t)\qquad  a.e.\ t.
\end{equation}
All these considerations imply that 
\[
\lim_n\lim_N P_n^N(a)=a\qquad in\ W^{1,2}([0,1],L^2(\ppi))
\]
for every $a\in W^{1,2}([0,1],L^2(\ppi))$ and thus to conclude it is sufficient to prove that $P_n^N(a)$ belongs to the $W^{1,2}([0,1],L^2(\ppi))$-closure of $\mathcal A$ for every $n,N\in \N$ and $a\in W^{1,2}([0,1],L^2(\ppi))$. 

It is clear by construction and \eqref{eq:pmder} that $P_n^N(a)=\sum_{i=1}^Ng_i\nchi_{E^n_i}$ for some $g_i\in W^{1,2}([0,1])$. Now for every $i=1,\ldots,N$ find $(g_{i,j})\subset{\rm LIP}([0,1])$ which $W^{1,2}([0,1])$-converges to $g_i$ and notice that
\[
\Big\|\sum_{i=1}^N(g_{i,j}-g_i)\nchi_{E^n_i}\Big\|^2_{W^{1,2}([0,1],L^2(\sppi))}=\sum_{i=1}^N\ppi(E^n_i)\|g_{i,j}-g_i\|_{W^{1,2}([0,1])}^2\quad\to\quad0\qquad\text{ as }j\to\infty.
\]
Since $\sum_{i=1}^Ng_{i,j}\nchi_{E^n_i}\in\mathcal A$ for every $j$, the proof is finished.
\end{proof}

\section{Parallel transport on {\sf RCD} spaces}
\subsection{Definition and basic properties of parallel transport}
\subsubsection{Definition and uniqueness}
%
%
We shall frequently use the fact that, since $\mathscr{H}^{1,2}(\ppi)$ is continuously embedded into $\mathscr{C}(\ppi)$ by Theorem \ref{thm:H12(ppi)_in_C(ppi)},  any vector field $V\in\mathscr{H}^{1,2}(\ppi)$ has pointwise values $V_t\in\e_t^*L^2(T\X)$ defined at every time $t\in[0,1]$.
\begin{definition}[Parallel transport]\label{def:pt}
Let $K\in\R$, $(\X,\sfd,\mm)$  an $\mathsf{RCD}(K,\infty)$ space and  $\ppi$ be a Lipschitz test plan on $\X$.  A \emph{parallel transport along $\ppi$} is an element $V\in\mathscr H^{1,2}(\ppi)$ such that ${\rm D}_\sppi V=0$.
\end{definition}
The linearity of the requirement ${\rm D}_\sppi V=0$ ensures that the set of parallel transports along $\ppi$ is a vector space. From Proposition \ref{prop:Leibniz_in_H} we deduce the following simple but crucial result:
\begin{proposition}[Norm preservation]\label{prop:normpres}
Let $V$ be a parallel transport along the Lipschitz test plan $\ppi$. Then $t\mapsto |V_t|^2\in L^1(\ppi)$ is constant.
\end{proposition}
\begin{proof}
We know from Corollary \ref{cor:contnorm}  that $t\mapsto |V_t|^2\in L^1(\ppi)$ is continuous. Hence the choice $W=V$ in Proposition \ref{prop:Leibniz_in_H} tells that such map is absolutely continuous with derivative given by
\[
\frac\d{\d t}|V_t|^2=2\la{\rm D}_\sppi V_t,V_t\ra=0,\qquad a.e.\ t.
\]
This is sufficient to conclude.
\end{proof}
Linearity and norm preservation imply uniqueness:
\begin{corollary}[Uniqueness of parallel transport]\label{cor:uni}
Let $\ppi$ be a Lipschitz test plan and $V_1,V_2$ two parallel transports along it such that for some $t_0\in[0,1]$ it holds $V_{1,t_0}=V_{2,t_0}$. Then $V_1=V_2$. 
\end{corollary}
\begin{proof}
Since ${\rm D}_\sppi(V_1-V_2)={\rm D}_\sppi V_1-{\rm D}_\sppi V_2=0$, we have that $V_1-V_2$ is a parallel transport and by assumption we know that $|V_{1,t_0}-V_{2,t_0}|=0$ $\ppi$-a.e.. Thus Proposition \ref{prop:normpres} above grants that for every $t\in[0,1]$ it holds $|V_{1,t}-V_{2,t}|=0$ $\ppi$-a.e., i.e.\ that $V_{1,t}=V_{2,t}$.
\end{proof}
\begin{remark}{\rm
We emphasize that the norm preservation property is a consequence of the Leibniz formula in Proposition \ref{prop:Leibniz_in_H}. We don't know if such formula holds for $V,W\in\mathscr W^{1,2}(\ppi)$ and this is why we defined the parallel transport as an element of $\mathscr H^{1,2}(\ppi)$ with null convective derivative, as opposed to an element of $\mathscr W^{1,2}(\ppi)$ with the same property.
}\fr\end{remark}
\subsubsection{Some consequences of existence of parallel transport}
In this section we assume existence of parallel transport along some/all Lipschitz test plans and see what can be derived from such assumption.

It will be convenient to recall the concept of \emph{base} of a module, referring to \cite{Gigli14} for a more detailed discussion. Let $\mu$ be a Borel measure on a Polish space $\Y$, $\mathscr M$ a $L^2(\mu)$-normed module, $v_1,\ldots,v_n\in\mathscr M$ and $E\subset\Y$ a Borel  set. Then the $v_i$'s are said to be \emph{independent} on $E$ provided for any $f_i\in L^\infty(\mu)$ we have
\[
\sum_{i}f_iv_i=0\qquad\Rightarrow\qquad f_i=0\quad\mu-a.e.\ on \ E
\] 
and \emph{generators} of $\mathscr M$ on $E$ provided $L^\infty(\mu)$-linear combinations of the $v_i$'s are dense in $\{\nchi_Ev\,:\,v\in\mathscr M\}$. If  $v_1,\ldots,v_n$ are both independent and generators of $\mathscr M$ on $E$ we say that they are a base of $\mathscr M$ on $E$ and in this case we say that the dimension of $\mathscr M$ on $E$ is $n$.

Recall that there always exists a (unique up to $\mu$-negligible sets) Borel partition $(E_i)_{i\in\N\cup\{\infty\}}$ of $\Y$, called \emph{dimensional decomposition of $\mathscr M$},  such that the dimension of $\mathscr M$ on $E_i$ is $i$ for every $i\in\N$ and for no Borel subset $F$ of $E_\infty$ with positive measure the dimension of $\mathscr M$ on $F$ is finite.

For a separable Hilbert module $\mathscr H$ on a space with finite measure $\mu$ one can always find an \emph{orthonormal base} $(v_n)_{n\in\N}$ i.e.\ a sequence whose $L^\infty$-linear combinations are dense in $\mathscr H$ and such that for some Borel partition $(E_n)_{n\in\N\cup\{\infty\}}$ the following hold:
\[
\begin{array}{lrll}
\forall n\in\N\cup\{\infty\}\text{ we have }&\quad\la v_i,v_j\ra\!\!\!&=\delta_{ij}&\qquad\mu-a.e.\ on \ E_n\ \forall i,j\in\N,\ i,j< n\\
\forall n\in\N\text{ we have }&\quad |v_i|\!\!\!&=0&\qquad\mu-a.e.\ on \ E_n\ \forall i\in\N,\ i\geq n
\end{array}
\]
and it is easily verified that if these hold, then necessarily the $E_n$'s form the dimensional decomposition of $\mathscr H$  
(the role of the assumption about finiteness of $\mu$ is to ensure that the $v_i$'s are elements of $\mathscr H$: their pointwise norm is in $L^\infty(\mu)$ and thus in general it may be not in $L^2(\mu)$).

Given such a base and $v\in\mathscr H$ there are functions $a_n\in L^2(\mu)$ such that
\begin{equation}
\label{eq:base}
v=\sum_{n\in\N}a_nv_n
\end{equation}
meaning that the sequence converges absolutely in $\mathscr H$. A choice for the $a_n$ is $a_n:=\la v,v_n\ra$ and for any two sequences $(a_n),(\tilde a_n)$ for which \eqref{eq:base} holds we have
\[
a_n=\tilde a_n\qquad\mu-a.e.\ on\ E_m,\quad \forall m\in\N\cup\{\infty\},\ m>n.
\]
With this said, it is easy to see that parallel transport sends bases into bases:
\begin{proposition}\label{prop:traspbase}
Let  $\ppi$ be a  Lipschitz test plan such that for every $t\in[0,1]$ and   $\bar V_t\in \e_t^*L^2(T\X)$ there exists a (necessarily unique) parallel transport $V$  along $\ppi$ such that $V_t=\bar V_t$.

Also, let $(E_n)_{n\in\N\cup\{\infty\}}$ be the dimensional decomposition of $\e_0^*L^2(T\X)$ and $(\bar V_n)\subset \e_0^*L^2(T\X)$, $n\in\N$ an orthonormal base of $ \e_0^*L^2(T\X)$ and denote by $t\mapsto V_{n,t}$ the parallel transport of $\bar V_n$ along $\ppi$.

Then for every $t\in[0,1]$ the partition  $(E_n)_{n\in\N\cup\{\infty\}}$ is also the dimensional decomposition of $\e_t^*L^2(T\X)$ and the set $(V_{n,t})$ is  an orthonormal base of $ \e_t^*L^2(T\X)$.
\end{proposition}
\begin{proof}
For every $t,s\in[0,1]$ consider the map sending $\bar V\in\e_t^*L^2(T\X)$ to $V_s\in\e_s^*L^2(T\X)$ where $V\in\mathscr H^{1,2}(\ppi)$ is the parallel transport such that $V_t=\bar V$. Proposition \ref{prop:normpres} ensures that this map preserves the pointwise norm. Since it is clearly linear, it is easily verified that it must be an isomorphism of $\e_t^*L^2(T\X)$ and $\e_s^*L^2(T\X)$.

The conclusions follow.
\end{proof}
We shall apply this result to show that, under the same assumptions, we have $\mathscr W^{1,2}(\ppi)=\mathscr H^{1,2}(\ppi)$:
\begin{proposition}[$\mathscr H=\mathscr W$]
Let  $\ppi$ be a  Lipschitz test plan such that for every $t\in[0,1]$ and   $\bar V_t\in \e_t^*L^2(T\X)$ there exists a (necessarily unique) parallel transport $V$  along $\ppi$ such that $V_t=\bar V_t$.

Then $\mathscr H^{1,2}(\ppi)=\mathscr W^{1,2}(\ppi)$.
\end{proposition}
\begin{proof}
Let $V\in \mathscr W^{1,2}(\ppi)$,  $(\bar V_i)\subset \e_0^*L^2(T\X)$, $i\in\N$,  an orthonormal base of $ \e_0^*L^2(T\X)$ and $t\mapsto V_{i,t}$ the parallel transport of $\bar V_i$ along $\ppi$. Then by Proposition \ref{prop:traspbase} above we see that
\begin{equation}
\label{eq:a}
V_t=\sum_{i\in\N}a_{i,t} V_{i,t}\qquad\text{where}\quad a_{i,t}:=\la V_t, V_{i,t}\ra\qquad a.e.\ t,
\end{equation}
being intended that the series converges absolutely in $\e_t^*L^2(T\X)$ for a.e.\ $t$. By Proposition \ref{prop:Leibniz_in_H} we see that $t\mapsto a_{i,t}$ is in $W^{1,1}([0,1],L^1(\ppi))$ with derivative given by
\begin{equation}
\label{eq:dera}
a_{i,t}'=\la {\rm D}_\sppi V_t,V_{i,t}\ra.
\end{equation} 
In particular, since $|V_{i,t}|\leq 1$ we see that  $a_{i,t},a'_{i,t}\in L^2([0,1],L^2(\ppi))$ and in turn this implies - by Proposition \ref{prop:carw1p} - that $t\mapsto a_{i,t}$ is in $W^{1,2}([0,1],L^2(\ppi))$.  This fact and point $(ii)$ in Proposition \ref{prop:exh} give that $(t\mapsto a_{i,t}V_{i,t})\in\mathscr H^{1,2}(\ppi)$ for every $i\in\N$ and therefore $(t\mapsto \sum_{i=0}^na_{i,t}V_{i,t})\in\mathscr H^{1,2}(\ppi)$ for every $n\in\N$. 

Hence to conclude it is sufficient to show that these partial sums are a $\mathscr W^{1,2}(\ppi)$-Cauchy sequence, as then it is clear from \eqref{eq:a} that the limit coincides with $V$. From \eqref{eq:a} and \eqref{eq:dera} we have that
\[
\sum_{i\in\N}\iint_0^1 |a_{i,t}|^2+ |a'_{i,t}|^2\,\d t\,\d\ppi=\|V\|_{\mathscr L^2(\sppi)}^2+\|{\rm D}_\sppi V\|_{\mathscr L^2(\sppi)}^2<\infty,
\]
hence the conclusion follows from the identity
\[
\Big\|\sum_{i=n}^ma_iV_i\Big\|_{\mathscr W^{1,2}(\sppi)}^2=\Big\|\sum_{i=n}^ma_iV_i\Big\|_{\mathscr L^{2}(\sppi)}^2+\Big\|\sum_{i=n}^ma_i'V_i\Big\|_{\mathscr L^{2}(\sppi)}^2=\sum_{i=n}^m\iint_0^1|a_{i,t}|^2+|a_{i,t}'|^2\,\d t\,\d\ppi.
\]
\end{proof}
We shall now prove that if the parallel transport exists along all Lipschitz test plans, then the dimension of $\X$, intended here as the dimension of the tangent module, must be constant. We shall use the following simple lemma (for simplicity we state it for Hilbert modules, but in fact the same holds for general ones):
\begin{lemma}\label{le:dimpb} Let $(\X,\sfd_{\X} ,\mm_\X)$ and $(\Y,\sfd_{\Y} ,\mm_\Y)$ be two metric measure spaces, $\varphi:\Y\to\X$ of bounded compression and $\mathscr H$ an Hilbert module on $\X$. Then for every $E\subset\X$ Borel we have that the dimension of $\mathscr H$ on $E$ is $n$ if and only if the dimension of $\varphi^*\mathscr H$ on $\varphi^{-1}(E)$ is $n$.
\end{lemma}
\begin{proof} From the well-posedness of the definition of dimension we see that it is sufficient to prove the \emph{only if}. Thus let $v_0,\ldots,v_{n-1}$ be an orthonormal
base of $\mathscr H$ on $E$, so in particular $\la v_i,v_j\ra=\delta_{ij}$ $\mm_\X$-a.e.\ on $E$.

From the fact that  $v_0,\ldots,v_{n-1}$ generate $\mathscr H$ on $E$ and the very definition of pullback we deduce that  $\varphi^*v_0,\ldots,\varphi^*v_{n-1}$ generate $\varphi^*\mathscr H$ on $\varphi^{-1}(E)$.

To see that they are independent, let $f_0,\ldots,f_{n-1}\in L^\infty(\Y)$ be such that $\sum_if_i\varphi^*v_i=0$ on $\varphi^{-1}(E)$. Then it holds
\[
0=\Big|\sum_if_i\varphi^*v_i\Big|^2=\sum_{i,j}f_if_j\la \varphi^*v_i,\varphi^*v_j\ra=\sum_{i,j}f_if_j\la v_i,v_j\ra\circ\varphi=\sum_if_i^2\qquad\mm_\Y-a.e.\ on\ \varphi^{-1}(E),
\]
and therefore $f_i=0$ $\mm_\Y$-a.e.\ on $\varphi^{-1}(E)$ for every $i=0,\ldots,n-1$. 
\end{proof}
We now prove that the dimension is constant:
\begin{theorem}[From parallel transport to constant dimension]\label{thm:constdim}
Let $(\X,\sfd,\mm)$ be a $\RCD(K,\infty)$ space such that for any Lipschitz test plan $\ppi$, any $t\in[0,1]$ and any $\bar V\in \e_t^*L^2(T\X)$ there exists the parallel transport $V$  along $\ppi$ such that  $V_t=\bar V$.

Then the tangent module has constant dimension, i.e.\ in its dimensional decomposition $(E_i)_{i\in\N\cup\{\infty\}}$ one of the $E_i$'s has full measure and the others are negligible. 
\end{theorem}
\begin{proof}
We argue by contradiction, thus we shall assume that for some $i, j\in\N\cup\{\infty\}$, $i\neq j$, we have $\mm(E_i),\mm(E_j)>0$. Let $F_0\subset E_i$ and $F_1\subset E_j$ be bounded, with positive and finite measure, consider
\[
\mu_0:=\mm(F_0)^{-1}\mm\restr{F_0}\qquad\qquad\mu_1:=\mm(F_1)^{-1}\mm\restr{F_1},
\]
let $\ppi$ be the unique optimal geodesic plan connecting them and recall that it is a test plan (see \cite{GigliRajalaSturm13}). Since $\ppi(\e_0^{-1}(F_0))=\mu_0(F_0)=1$ and the dimension of $L^2(T\X)$ on $F_0$ is $i$, by Lemma \ref{le:dimpb} above we see that for the dimensional decomposition $(\tilde E^0_n)_{n\in\N\cup\{\infty\}}$  of $\e_0^*L^2(T\X)$ we have $\ppi(\tilde E^0_i)=1$ and $\ppi(\tilde E^0_k)=0$ for every $k\neq i$. Similarly, for the 
 the dimensional decomposition $(\tilde E^1_n)_{n\in\N\cup\{\infty\}}$  of $\e_1^*L^2(T\X)$ we have $\ppi(\tilde E^1_j)=1$ and $\ppi(\tilde E^1_k)=0$ for every $k\neq j$. In particular we have
 \begin{equation}
\label{eq:dimdiv}
\ppi(\tilde E^0_i\Delta \tilde E^1_i)=\ppi(\tilde E^0_i)=1>0.
\end{equation}
Now notice that from basic considerations about optimal transport we have that $\ppi$ is concentrated on geodesics starting from $F_0$ and ending in $F_1$. The constant speed of any such geodesic is bounded from above by $\sup_{x\in F_0,y\in F_1}\sfd(x,y)<\infty$ so that $\ppi$ is a Lipschitz test plan. Therefore from Proposition \ref{prop:traspbase} we know that the dimensional decomposition of $\e_t^*L^2(T\X)$ does not depend on $t$. This however contradicts \eqref{eq:dimdiv}, hence the proof is completed.
\end{proof}

\subsection{Existence of the parallel transport in a special case}\label{se:ex}
It is unclear to us whether on general $\RCD$ spaces the parallel transport exists or not. Aim of this section it to show at least that the theory we propose is not empty, i.e.\ that under suitable assumptions on the space, the parallel transport exists. We won't insist in trying to make such assumptions as general as possible (for instance, the `good base' defined below could consist in different vector fields on different open sets covering our space) as our main concern is only to show that in some circumstances our notion of parallel transport can be shown to exist.

We shall work with spaces admitting the following sort of base for the tangent module: 
\begin{definition}[Good base]\label{def:good_basis}
Let $({\rm X},\sfd,\mm)$ be a given $\mathsf{RCD}(K,N)$ space, for some $K\in\R$ and $N\in(1,\infty)$.
Let us denote by ${(A_k)}_{k=1}^n$ the dimensional decomposition of $\rm X$.
Then a family $\mathcal{W}=\{w_1,\ldots,w_n\}\subseteq H^{1,2}_C(T{\rm X})$ of Sobolev vector fields 
on $\rm X$ is said to be a \emph{good basis} for $L^2(T{\rm X})$ provided there exists
$M>0$ such that the following properties are satisfied:
\begin{itemize}
\item[(i)] For any $k=1,\ldots,n$, we have that $w_1,\ldots,w_k$ 
constitute a basis for $L^2(T{\rm X})$ on $A_k$ and
\begin{equation}\label{eq:quasi_orth_Sobolev_basis}
\left\{\begin{array}{ll}
|w_i|\in( M^{-1},M),\\
\big|\langle w_i, w_j\rangle\big|<\frac{1}{M^2k}
\end{array}\right.
\mm\mbox{-a.e.\ in }A_k,\quad\mbox{ for every }i,j=1,\ldots,k\mbox{ with }i\neq j,
\end{equation}
\item[(ii)] We have
\begin{equation}\label{eq:bdd_basis}
{|\nabla w_i|}_{\mathsf{HS}}\leq M\quad
\mm\text{-a.e.\ in }{\rm X},\qquad\forall i=1,\ldots,n.
\end{equation}
\end{itemize}
\end{definition}
Let us notice that the `hard' assumption here is given by point $(ii)$ (perhaps coupled with the lower bound in $(i)$), which imposes an $L^\infty$ bound on covariant derivative, when in our setting $L^2$-ones are more natural (compare with Theorem \ref{thm:sobbase}). Let us mention in particular that for spaces admitting a good base it is not hard to prove, regardless of parallel transport, that the dimension is constant (see Proposition \ref{prop:constdim} below).

Let us start the technical work with the following simple lemma:
\begin{lemma}\label{lem:coordinates_Linfty}
Let $\mathscr H$ be a Hilbert module on $\X$, $A\subset\X$ be a Borel set and $M>1$. Also, for $k\in\N$ let $w_1,\ldots,w_k\in \mathscr H$ be such that
\[
\left\{\begin{array}{ll}
|w_i|\in(M^{-1},M)\\
\big|\la w_i,w_j\ra\big|\leq\frac{1}{M^2k}
\end{array}\right.\mbox{ hold }\mm\mbox{-a.e.\ in }A,
\quad\mbox{ for every }i,j=1,\ldots,k\mbox{ with }i\neq j.
\]
For   $h_1,\ldots,h_k\in L^0(\mm)\restr{A}$ put $w:=\sum_{i=1}^k h_i\,w_i\in\mathscr H^0$ (the $L^0$-completion of $\mathscr H$).

Then it holds
\begin{equation}
\label{eq:almort}
\frac{1}{M^2k}\sum_{i=1}^k|h_i|^2\leq |w|^2\leq M^2 k\sum_{i=1}^k|h_i|^2\qquad \mm-a.e.\ on\ A 
\end{equation}
and in particular $w\in\mathscr H\restr A$ if and only if $h_i\in L^2(\mm)\restr A$ for every $i=1,\ldots,k$.
\end{lemma}
\begin{proof} For the second  in \eqref{eq:almort} we notice that $\la w_i,w_j\ra\leq M^2$ $\mm$-a.e.\ on $A$ for every $i,j$, thus
\[
|w|^2=\big|\sum_{i=1}^kh_i\,w_i\big|^2=\sum_{i,j=1}^kh_ih_j\la w_i,w_j\ra\leq M^2 \sum_{i,j=1}^k\frac12|h_i|^2+\frac12|h_j|^2= M^2k\sum_{i=1}^k|h_i|^2.
\]
For the first inequality we recall that  $|w_i|>M^{-1}$  and $\la w_i,w_j\ra\geq -\frac{1}{M^2k}$ for $i\neq j $ $\mm$-a.e.\ on $A$ to deduce 
\[
\begin{split}
|w|^2&=\big|\sum_{i=1}^kh_i\,w_i\big|^2=\sum_{i=1}^k|h_i|^2|w_i|^2+\sum_{i\neq j}h_ih_j\la w_i,w_j\ra\geq\frac1{M^2}\sum_{i=1}^k|h_i|^2-\frac{1}{M^2k}\sum_{i\neq j}|h_ih_j|\\
&\geq\frac1{M^2}\sum_{i=1}^k|h_i|^2-\frac{1}{M^2k}\sum_{i\neq j}\frac12|h_i|^2+\frac12|h_j|^2=\frac1{M^2k}\sum_{i=1}^k|h_i|^2.
\end{split}
\]
\end{proof}
The constant dimension now easily follows from the lemma and the fact that if a good base exists, then there is another one for which the functions $\la w_i,w_j\ra$ are Lipschitz, as shown in the proof of the following proposition.
\begin{proposition}\label{prop:constdim}
Let $(\X,\sfd,\mm)$ be a ${\RCD}(K,N)$ space admitting a good base for $L^2(T\X)$. Then the tangent module has constant dimension, i.e.\ in its dimensional decomposition $(E_k)_{k=1}^n$ one of the $E_k$'s has full measure and the others are negligible. 
\end{proposition}
\begin{proof}
Let $k\in\N$ be the maximal index such that $\mm(E_k)>0$ (its existence follows from the finiteness results in \cite{Han14} and \cite{GP16}). To conclude it is enough to show that on a neighbourhood of $E_k$ the tangent module has dimension $\geq k$. Let $(w_i)_{i=1}^k$ be a good base, $f\in C^\infty_c(\R)$  be such that $f(z)=z$ for every $z\in[0,M]$ and consider the vector fields $\tilde w_i:=f(|w_i|^2)w_i$. Notice that by \cite{Gigli14} we know that $f(|w_i|^2)\in W^{1,2}(\X)$ with $\nabla f(|w_i|^2)=2f'(|w_i|^2)\nabla w_i(\cdot,w_i)$, hence by \eqref{eq:bdd_basis} and the choice of $f$ we see that $f(|w_i|^2)$ is bounded with bounded gradient. It follows (see \cite{Gigli14}) that $\tilde w_i\in H^{1,2}_C(T\X)$ with
\[
\nabla\tilde w_i=\nabla f(|w_i|^2)\otimes w_i+f(|w_i|^2)\nabla w_i,
\]so that from the expression of $\nabla f(|w_i|^2)$ we see that $\tilde w_i$ is bounded with bounded covariant derivative. Hence   $g_{i,j}:=\la\tilde w_i,\tilde w_j\ra$ belongs to $W^{1,2}(\X)$ (see \cite{Gigli14}) and is bounded with bounded gradient as well. By the Sobolev-to-Lipschitz property (see \cite{AmbrosioGigliSavare11-2}, \cite{Gigli13}, \cite{Gigli13over}) we deduce that $g_{i,j}$ has a Lipschitz, in particular continuous, representative. By construction, the bounds \eqref{eq:quasi_orth_Sobolev_basis} hold on $E_k$ for the $\tilde w_i$'s, hence the continuity of $g_{i,j}$ grants that they hold also on  some neighbourhood of $E_k$, and by Lemma \ref{lem:coordinates_Linfty} above this is sufficient to conclude that the vector fields $\tilde w_i$ are independent - by the first in \eqref{eq:almort} -  on such neighbourhood, thus concluding the proof.
\end{proof}

We now prove  existence of the parallel transport for the class of those
\textsf{RCD} spaces that admit a good base for their tangent module. In the proof we shall use, for simplicity, the fact just proved that the dimension must be constant, but in fact the same argument works even without knowing a priori this fact (albeit since constant dimension follows so directly from the existence of a good base, this remark is perhaps irrelevant).
\begin{theorem}[Existence of the parallel transport]\label{thm:existence_PT} Let $(\X,\sfd,\mm)$ be any $\mathsf{RCD}(K,N)$ space, $K\in\R$ and $N\in(1,\infty)$, that admits a good base. Let $\ppi$ be a Lipschitz test plan on $\rm X$ and  fix $\overline{V}\in\e_0^*L^2(T{\rm X})$. 

Then there exists the parallel transport
$V\in\mathscr{H}^{1,2}(\ppi)$ along $\ppi$ such that $V_0=\overline{V}$.
\end{theorem}
\begin{proof} We know by Proposition \ref{prop:constdim} that the dimension of the tangent module must be constant. From this fact and Definition \ref{def:good_basis} we know that  for some $n$ we have  $w_1,\ldots,w_n\in H^{1,2}_C(T\X)$ for which \eqref{eq:quasi_orth_Sobolev_basis} and \eqref{eq:bdd_basis} hold on  $\X$. Put $W_{i,t}:=\e_t^*w_i$. By point $(i)$ in Proposition \ref{prop:exh} we have that $W_i\in\mathscr H^{1,2}(\ppi)$ with 
\[
{\rm D}_\sppi W_{i,t}={\rm Cov}_t(w_i)
\]
and therefore from \eqref{eq:b3}, \eqref{eq:bdd_basis} and the assumption that $\ppi$ is Lipschitz we get
\begin{equation}
\label{eq:derwin}
|{\rm D}_\sppi W_{i,t}|\leq M\sfL(\ppi).
\end{equation}
Also, by the defining property of the pullback map and from \eqref{eq:quasi_orth_Sobolev_basis} we have  that for every $t\in[0,1]$ it holds
\[
\left\{\begin{array}{ll}
|W_{i,t}|\in( M^{-1},M),\\
\big|\langle W_{i,t}, W_{j,t}\rangle\big|\leq{(M^2k)^{-1}}
\end{array}\right.
\ppi\mbox{-a.e.},\quad\mbox{ for every }i,j=1,\ldots,n\mbox{ with }i\neq j,
\]
thus Lemma \ref{lem:coordinates_Linfty} grants that there are functions  $\overline{g}_1,\ldots,\overline{g}_n\in L^2(\ppi)$
such that $\overline{V}=\sum_{i=1}^n\overline{g}_i\,\e_0^*w_i$. A similar argument applied to the pullback of the map $\e$ (recall Proposition \ref{prop:l2pb} and the definition \eqref{eq:mult_L2(pi)}) and based on the bound \eqref{eq:derwin} shows that there are functions $H_{i,j}\in L^\infty(\ppi\times\mathcal L_1)$ such that
\begin{equation}
\label{eq:dercomp}
{\rm D}_\sppi W_{i,t}=\sum_jH_{i,j,t}W_{j,t}\qquad a.e.\ t. 
\end{equation}
It will be technically convenient to fix once and for all Borel representatives, still denoted by $H_{i,j}$, of these functions such that
\begin{equation}
\label{eq:normsup}
\sup_{\gamma,t}|H_{i,j,t}(\gamma)|=\|H_{i,j}\|_{L^\infty(\sppi\times\mathcal L_1)}\qquad\forall i=1,\ldots,n.
\end{equation}
We shall look for a parallel transport of the form $V:=\sum_ig_iW_i$ with $g_i\in AC^2([0,1],L^2(\ppi))$. Notice that Lemma \ref{lem:coordinates_Linfty} grants that any such $V$ belongs to $\mathscr H^{1,2}(\ppi)$ with 
\[
\begin{split}{\rm D}_\sppi V_t
&\stackrel{\eqref{eq:leiba}}=
\sum_{i=1}^n g'_{i,t}\,W_{i,t}+\sum_{i=1}^n g_{i,t}\,{\rm D}_\sppi W_{i,t}
\stackrel{\eqref{eq:dercomp}}=\sum_{i=1}^n g'_{i,t}\,W_{i,t}+\sum_{i,j=1}^n g_{i,t}\,H_{i,j,t}\,W_{j,t}\\
&\stackrel{\phantom{\eqref{eq:leiba}}}=
\sum_{i=1}^n\bigg(g'_{i,t}+\sum_{j=1}^n H_{j,i,t}\,g_{j,t}\bigg)W_{i,t}
\qquad\mbox{ for }\mathcal{L}^1\mbox{-a.e.\ }t\in[0,1].
\end{split}
\]
Hence our $V$ is the desired parallel transport if and only if the functions $g_1,\ldots,g_n$ solve the system
\begin{equation}
\label{eq:perpt}
\left\{\begin{array}{rll}
g_{i,0}\!\!\!&=\overline g_i&\qquad\forall i=1,\ldots,n,\\
g_{i,t}'+\sum_{j=1}^n H_{j,i,t}\,g_{j,t}\!\!\!&=0&\qquad\forall i=1,\ldots,n
\mbox{ and for }\mathcal{L}^1\mbox{-a.e.\ }t\in[0,1].
\end{array}\right.
\end{equation}
To solve this system we shall apply Theorem \ref{cor:linear_ODE_diff} to the Banach (in fact Hilbert) space $\B:=[L^2(\ppi)]^n$ equipped with the norm
\[
\|f\|_\B^2:=\sum_{i=1}^n\int |f_i|^2\,\d\ppi\qquad\forall f=(f_1,\ldots,f_n)\in\B.
\]
For every $t\in[0,1]$ define $\lambda_t\in{\rm End}(\B)$ as
\[
(\lambda_tf)_i:=-\sum_{j=1}^n H_{j,i,t}f_j\qquad\qquad\forall i=1,\ldots,n\ and\ f=(f_1,\ldots,f_n)\in\B,
\]
so that the system \eqref{eq:perpt} can be rewritten as
\[
\left\{\begin{array}{ll}
g_{0}\!\!\!\!&=\overline g,\\
g'_t\!\!\!\!&=\lambda_tg_t\qquad\text{for }a.e.\  t\in[0,1],
\end{array}\right.
\]
where $\overline g:=(\overline g_1,\ldots,\overline g_n)$. Theorem \ref{cor:linear_ODE_diff} grants that a solution in ${\rm LIP}([0,1],\B)\subset AC^2([0,1],\B)$ exists provided the $\lambda_t$'s are equibounded and $t\mapsto\lambda_tf$ is strongly measurable for every $f\in\B$. The former follows from
\[\begin{split}
{\big\|\lambda_tf\big\|}_\B^2&=
\sum_{i=1}^n\bigg\|\sum_{j=1}^n H_{j,i,t}f_j\bigg\|_{L^2(\sppi)}^2
\leq n\sum_{i,j=1}^n{\big\|H_{j,i,t}f_j\big\|}_{L^2(\sppi)}^2\\
&\leq n\max_{i,j,t}{\big\|H_{j,i,t}\big\|}_{L^\infty(\sppi)}^2\sum_{i,j=1}^n{\|f_j\|}_{L^2(\sppi)}^2
\stackrel{\eqref{eq:normsup}}\leq n^2\max_{i,j}{\|H_{i,j}\|}_{L^\infty(\sppi\times\mathcal{L}_1)}^2\,{\|f\|}_\B^2.
\end{split}\]
For the latter, notice that since $\B$ is separable it is sufficient to prove that for any $f\in\B$ the map $t\mapsto \lambda_tf\in\B$ is weakly measurable. Since $\B$ is also Hilbert we need to show that for every $f,g\in\B$ the map $t\mapsto\la\lambda_tf,g\ra_\B\in\R$ is measurable. Since we have
\[
\la\lambda_tf,g\ra_\B=-\sum_{i,j=1}^n\int H_{j,i,t}f_jg_i\,\d\ppi,
\] 
the conclusion follows from Fubini's theorem.
\end{proof}
\appendix
\section{Sobolev base of the tangent module}\label{app:Sob_basis}
In this appendix we show that one can always build a base of the tangent module of a $\RCD$ space which has at least Sobolev regularity, as opposed to just $L^2$ regularity. The basic idea used in the construction is based on the observation that `being a base' is a non-linear requirement. Technically speaking, the crucial argument is contained in the following lemma:
\begin{lemma}\label{le:approx_Sobolev}
Let $(\X,\sfd,\mm)$ be a $\RCD(K,\infty)$ space and $w\in L^2(T\X)$.

Then there exists $v\in H^{1,2}_C(T\X)$ such that $\la v,w\ra\neq 0$ holds $\mm$-a.e.\ on $\{|w|\neq 0\}$.
\end{lemma}
\begin{proof} We can assume $w\neq 0$ or otherwise there is nothing to prove; then replacing if necessary $w$ with $(\nchi_{\{|w|\leq 1\}}+\nchi_{\{|w|>1\}}|w|^{-1})w$ we can also assume that $|w|\leq 1$ $\mm$-a.e.. Let $(w_n)\subset{\rm TestV}(\X)$ be $L^2(T\X)$-converging to $w$ and $\tilde\mm$  a Borel probability measure on $\X$ such that $\mm\ll\tilde\mm\leq\mm$. Then $\la w_n,w\ra\to|w|^2$ in $L^2(\X,\tilde\mm)$ and thus
\begin{equation}
\label{eq:mn}
m_n:=\tilde\mm\big(\{|\la w_n,w\ra|>0\}\big)\qquad\to \qquad m_\infty:=\tilde \mm\big(\{|w|>0\}\big).
\end{equation}
We now observe that
\begin{equation}
\label{eq:zero}
\begin{split}
\text{for every $v,\tilde w\in L^2(T\X)$ and $a>0$ there is $b\in(0,a)$}&\text{ such that}\\
\tilde\mm\Big(\big\{|\la \tilde w,w\ra|>0\big\}\cap \big\{\la v+b\tilde w,w\ra=0\big\}\Big)=0.
\end{split}
\end{equation}
Indeed, putting for brevity $E_b:=\big\{|\la\tilde w,w\ra|>0\big\}\cap \big\{\la v+b\tilde w,w\ra=0\big\}$ we have
\[
\tilde\mm\big(E_b\cap E_{b'}\big)\leq \tilde\mm\Big(\big\{|\la\tilde w,w\ra|>0\big\}\cap \big\{(b-b')\la \tilde w,w\ra=0\big\}\Big)=0\qquad\forall b\neq b',
\]
so that the claim follows from the finiteness of $\tilde\mm$ and the fact that the interval $(0,a)$ is uncountable.

Now put $\alpha_n:=\||w_n|\|_{L^\infty(\X)}+\|w_n\|_{W^{1,2}_C(T\X)}$ and recursively define  decreasing sequences $(\beta_n),(\gamma_n)\subset(0,\infty)$  such that $\beta_1=1$ and for every $n\in\N$ we have
\[
3\beta_{n+1}\leq\gamma_{n+1}\leq {\beta_n} \quad \text{and for}\quad E_{n}:=\Big\{\big|\la\sum_{i=1 }^n\frac{\beta_i}{\alpha_i}w_i,w\ra\big| \geq \gamma_{n+1}\Big\}\quad\text{it holds}\quad \tilde\mm(E_{n})\geq \frac{m_n}{1+\frac1n}.
\]
To see that this is possible, let $\beta_1=1$, and notice that trivially $\big\{|\la\frac{\beta_1}{\alpha_1}w_1,w\ra|>0\big\}=\big\{|\la w_1,w\ra|>0\big\}$ so that for $\gamma_{2}\in(0,\beta_1)$ sufficiently small the above holds. Now assume that $\beta_{n-1}$ and $\gamma_{n}$ have been found, use \eqref{eq:zero} for $v:=\sum_{i=1 }^{n-1}\frac{\beta_i}{\alpha_i}w_i$, $\tilde w:=w_{n}$ and $a:=\gamma_{n}/3$ to find $\beta_{n}:=b<\gamma_{n}/3$ such that $\tilde\mm\big(\big\{|\la\sum_{i=1 }^{n}\frac{\beta_i}{\alpha_i}w_i,w\ra|>0\big\}\big)\geq\tilde\mm\big(\big\{|\la w_n,w\ra|>0\big\}\big)=m_n$. Hence for $\gamma_{n+1}\in(0,\beta_{n})$ sufficiently small the claim holds.

We claim that the vector $v:=\sum_{i\geq 1}\frac{\beta_i}{\alpha_i}w_i$ satisfies the conclusion of the statement and start observing that $\beta_i\leq 3^{-i}$ and thus $\|\frac{\beta_i}{\alpha_i}w_i\|_{W^{1,2}_C(T\X)}\leq3^{-i}\|{\alpha_i}^{-1}w_i\|_{W^{1,2}_C(T\X)}\leq 3^{-i}$ by definition of $\alpha_i$. Hence the series converges in $W^{1,2}_C(T\X)$, so that $v$ is well defined and belongs to $H^{1,2}_C(T\X)$. Now notice that by construction and \eqref{eq:mn} we have $\tilde\mm(E_n)\to m_\infty$ and $\tilde\mm\big(E_n\setminus\{|w|>0\}\big)=0$, so that $\tilde\mm\big(\{|w|>0\}\setminus \cup_nE_n\big)=0$. Hence to conclude it is sufficient to show that for every $n\geq1$ it holds $\la v,w\ra\neq 0$ $\tilde \mm$-a.e.\ on $E_n$. Fix such $n$, let $m>n$ and observe that by definition of the $\alpha_i$'s and $\beta_i$'s we have
\[
\Big|\big\langle\frac{\beta_m}{\alpha_m}w_m,w\big\rangle\Big|\leq 3^{n-m+1}\beta_{n+1}\big|\la \alpha_m^{-1}w_m,w\ra\big|\leq 3^{n-m+1}\beta_{n+1},\qquad\tilde\mm-a.e.,
\]
so that $\big|\sum_{m>n}\la\frac{\beta_m}{\alpha_m}w_m,w\ra\big|\leq\frac32{\beta_{n+1}}\leq\frac12\gamma_{n+1}$. On the other hand by construction we have that $\big|\sum_{i=1}^n\la\frac{\beta_i}{\alpha_i}w_i,w\ra\big|\geq\gamma_{n+1}$ holds $\tilde\mm$-a.e.\ on $E_n$, granting that $|\la v,w\ra|\geq\frac12\gamma_{n+1}$ $\tilde\mm$-a.e.\ on $E_n$. 
\end{proof}

By repeatedly applying Lemma \ref{le:approx_Sobolev}, we can find a family of
$H^{1,2}_C(T{\rm X})$-Sobolev generators of the tangent module
on any $\mathsf{RCD}(K,\infty)$ space $\rm X$, as follows:
\begin{theorem}[Sobolev base of the tangent module]\label{thm:sobbase}
\label{thm:Sobolev_basis}
Let $({\rm X},\sfd,\mm)$ be an ${\sf RCD}(K,\infty)$ space, for some constant $K\in\R$.
Suppose that the dimensional decomposition of $\rm X$ is given by ${(A_n)}_{n\in\N}$.
Then there exists a sequence of vector fields ${(v_n)}_{n\geq 1}\subseteq H^{1,2}_C(T{\rm X})$ such that
\[
v_1,\ldots,v_n\mbox{ is a basis for }L^2(T{\rm X})\mbox{ on }A_n,\quad\mbox{ for every }n\in\N^+.
\]
\end{theorem}
\begin{proof}
The thesis can be equivalently rewritten in the following way:
\begin{equation}\label{eq:Sobolev_basis_alternative}
v_1,\ldots,v_n\mbox{ are independent on }{\bigcup}_{k\geq n}A_k,\quad\mbox{ for every }n\in\N^+.
\end{equation}
We build the sequence ${(v_n)}_n$ by means of a recursive argument.
First of all, choose a vector field $w\in L^2(T{\rm X})$ such that
$0<|w|\leq 1$ $\mm$-a.e.\ in $\bigcup_{k\geq 1}A_k$, then pick
$v_1\in H^{1,2}_C(T{\rm X})$ such that $\la v_1,w\ra\neq 0$ $\mm$-a.e.\ in $\bigcup_{k\geq 1}A_k$,
whose existence is granted by Lemma \ref{le:approx_Sobolev}. Thus in particular we have
$|v_1|>0$ $\mm$-a.e.\ in $\bigcup_{k\geq 1}A_k$, proving \eqref{eq:Sobolev_basis_alternative} for $n=1$.
Now suppose to have already found $v_1,\ldots,v_n$ satisfying the required property.
It can be easily seen that there exists $w\in L^2(T{\rm X})$ such that
$\la v_1,w\ra=\ldots=\la v_n,w\ra=0$ and $0<|w|\leq 1$
hold $\mm$-a.e.\ in the set $\bigcup_{k>n}A_k$.
Hence Lemma \ref{le:approx_Sobolev} ensures the existence of a vector field
$v_{n+1}\in H^{1,2}_C(T{\rm X})$ such that $\la v_{n+1},w\ra\neq 0$ $\mm$-a.e.\ in $\bigcup_{k>n}A_k$.

Now take any $f_1,\ldots,f_{n+1}\in L^\infty(\mm)$ such that $\sum_{i=1}^{n+1}f_i\,v_i=0$
$\mm$-a.e.\ in $\bigcup_{k>n}A_k$, thus one has $f_{n+1}\la v_{n+1},w\ra
=\sum_{i=1}^{n+1}f_i\la v_i,w\ra=0$ $\mm$-a.e.\ in $\bigcup_{k>n}A_k$, from which we can deduce that
$f_{n+1}=0$ holds $\mm$-a.e.\ in $\bigcup_{k>n}A_k$. Therefore $\sum_{i=1}^n f_i\,v_i=0$ $\mm$-a.e.\ in
$\bigcup_{k>n}A_k$ and accordingly also $f_1=\ldots=f_n=0$ $\mm$-a.e.\ in $\bigcup_{k>n}A_k$,
as a consequence of the independence of $v_1,\ldots,v_n$.
This grants that the vector fields $v_1,\ldots,v_{n+1}$ are independent on $\bigcup_{k>n}A_k$,
proving \eqref{eq:Sobolev_basis_alternative} for $n+1$. The thesis is then achieved.
\end{proof}

\end{document}